\documentclass[10pt]{amsart}

\usepackage[mathscr]{euscript}
\DeclareFontFamily{OT1}{rsfs}{}
\DeclareFontShape{OT1}{rsfs}{n}{it}{<-> rsfs10}{}
\DeclareMathAlphabet{\curly}{OT1}{rsfs}{n}{it}

\makeatletter
\newcommand{\eqnum}{\refstepcounter{equation}\textup{\tagform@{\theequation}}}
\makeatother

\newcommand\beq[1]{\begin{equation}\label{#1}}
	\newcommand\eeq{\end{equation}}
\newcommand\beqa{\begin{eqnarray*}}
	\newcommand\eeqa{\end{eqnarray*}}

\title[Derived categories of Quot schemes]{Derived categories of Quot schemes 
of zero-dimensional quotients on curves}
\date{}
\author{Yukinobu Toda}

\usepackage{amscd}
\usepackage{amsmath}
\usepackage{amssymb}
\usepackage{amsthm}
\usepackage{float}
\usepackage{graphicx}
\usepackage{tikz}
\usepackage{xypic}
\usetikzlibrary{cd}
\usepackage{here}
\usetikzlibrary{intersections, calc, arrows.meta}

\usepackage{youngtab}

\usepackage{array}
\usepackage{amscd}
\usepackage[all]{xy}

\usepackage{tabulary}
\usepackage{booktabs}

\usepackage{amssymb, paralist, xspace, url, amscd, euscript, mathrsfs, stmaryrd}

\DeclareFontFamily{U}{rsfs}{%
	\skewchar\font127}
\DeclareFontShape{U}{rsfs}{m}{n}{%
	<-6>rsfs5<6-8.5>rsfs7<8.5->rsfs10}{}
\DeclareSymbolFont{rsfs}{U}{rsfs}{m}{n}
\DeclareSymbolFontAlphabet
{\mathrsfs}{rsfs}
\DeclareRobustCommand*\rsfs{%
	\@fontswitch\relax\mathrsfs}

\theoremstyle{plain}
\newtheorem{thm}{Theorem}[section]
\newtheorem{prop}[thm]{Proposition}
\newtheorem{lem}[thm]{Lemma}

\newtheorem{defi}[thm]{Definition}
\newtheorem{rmk}[thm]{Remark}
\newtheorem{cor}[thm]{Corollary}

\newtheorem{prop-defi}[thm]{Proposition-Definition}
\newtheorem{thm-defi}[thm]{Theorem-Definition}
\newtheorem{lem-defi}[thm]{Lemma-Definition}

\newtheorem{exam}[thm]{Example}
\newtheorem{algrho}[thm]{Algorhithm}

\newcommand{\ssslash}{/\!\!/}

\newcommand{\cC}{\mathcal{C}}
\newcommand{\dD}{\mathcal{D}}
\newcommand{\eE}{\mathcal{E}}
\newcommand{\fF}{\mathcal{F}}
\newcommand{\gG}{\mathcal{G}}
\newcommand{\hH}{\mathcal{H}}

\newcommand{\lL}{\mathcal{L}}
\newcommand{\mM}{\mathcal{M}}
\newcommand{\nN}{\mathcal{N}}
\newcommand{\oO}{\mathcal{O}}
\newcommand{\pP}{\mathcal{P}}
\newcommand{\qQ}{\mathcal{Q}}

\newcommand{\sS}{\mathcal{S}}

\newcommand{\uU}{\mathcal{U}}

\newcommand{\wW}{\mathcal{W}}
\newcommand{\xX}{\mathcal{X}}
\newcommand{\yY}{\mathcal{Y}}
\newcommand{\zZ}{\mathcal{Z}}

\newcommand{\fM}{\mathfrak{M}}

\newcommand{\fU}{\mathfrak{U}}

\newcommand{\Supp}{\mathop{\rm Supp}\nolimits}
\newcommand{\Hom}{\mathop{\rm Hom}\nolimits}

\newcommand{\dR}{\mathbf{R}}

\newcommand{\Pic}{\mathop{\rm Pic}\nolimits}

\newcommand{\id}{\textrm{id}}

\newcommand{\ch}{\mathop{\rm ch}\nolimits}

\newcommand{\Ext}{\mathop{\rm Ext}\nolimits}
\newcommand{\Spec}{\mathop{\rm Spec}\nolimits}
\newcommand{\rank}{\mathop{\rm rank}\nolimits}
\newcommand{\Coh}{\mathop{\rm Coh}\nolimits}

\newcommand{\ev}{\mathop{\rm ev}\nolimits}

\newcommand{\us}{\mathchar`-\rm{us}}
\newcommand{\sss}{\mathchar`-\rm{ss}}

\newcommand{\cneq}{\mathrel{\raise.095ex\hbox{:}\mkern-4.2mu=}}
\newcommand{\eqcn}{\mathrel{=\mkern-4.5mu\raise.095ex\hbox{:}}}

\newcommand{\ext}{\mathop{\rm ext}\nolimits}

\newcommand{\Aut}{\mathop{\rm Aut}\nolimits}

\newcommand{\PPer}{\mathop{\rm Per}\nolimits}

\newcommand{\IC}{\mathop{\rm IC}\nolimits}

\newcommand{\DT}{\mathop{\rm DT}\nolimits}

\newcommand{\Sym}{\mathop{\rm Sym}\nolimits}

\newcommand{\modu}{\mathop{\rm mod}\nolimits}

\newcommand{\End}{\mathop{\rm End}\nolimits}

\newcommand{\Quot}{\mathop{\rm Quot}\nolimits}

\newcommand{\RHom}{\mathop{\dR\mathrm{Hom}}\nolimits}
\newcommand{\Ker}{\mathop{\rm Ker}\nolimits}

\newcommand{\GL}{\mathop{\rm GL}\nolimits}

\newcommand{\cl}{\mathop{\rm cl}\nolimits}

\newcommand{\MF}{\mathop{\rm MF}\nolimits}

\newcommand{\Crit}{\mathop{\rm Crit}\nolimits}

\newcommand{\wt}{\mathrm{wt}}

\newcommand{\qcoh}{\mathrm{qcoh}}

\newcommand{\inclusion}{\ar@<-0.3ex>@{^{(}->}[r]}
\newcommand{\upinclusion}{\ar@<-0.3ex>@{^{(}->}[u]}
\newcommand{\leinclusion}{\ar@<-0.3ex>@{^{(}->}[l]}
\newcommand{\doinclusion}{\ar@<-0.3ex>@{^{(}->}[d]}
\newcommand{\diasquare}{\ar@{}[rd]|\square}

\newcommand{\lgakko}{(\!(}
\newcommand{\rgakko}{)\!)}

\newcommand{\C}{\mathbb{C}^{\ast}}

\newcommand{\st}{\mathchar`-\rm{st}}

\newcommand{\dDT}{\mathcal{DT}}

\makeatletter
\renewcommand{\theequation}{%
	\thesection.\arabic{equation}}
\@addtoreset{equation}{section}
\makeatother

\begin{document}
	
	\begin{abstract}
		We prove the existence of semiorthogonal decompositions of 
		derived categories of Quot schemes 
		of zero-dimensional quotients on curves
		in terms of derived categories of symmetric products of curves. 
		The above result is a categorical analogue of a similar formula 
		for the class of Quot schemes in the Grothendieck ring of 
		varieties by Bagnarol-Fantechi-Perroni. 
		It is a special case of a more 
		general Quot formula of relative dimension one, 
		which is regarded as a Bosonic counterpart 
		of the Quot formula conjectured by Jiang and proved by the author. 
		The proof involves categorical wall-crossing formula 
		for framed one loop quiver, which itself 
		is motivated and has applications to 
		categorical wall-crossing formula of Donaldson-Thomas 
		invariants. 
	
	\end{abstract}
	
	\maketitle
	

	\section{Introduction}
	\subsection{Derived categories of Quot schemes over curves}
	For $(r, d) \in \mathbb{Z}_{\ge 0}^2$, let 
	$\mathrm{Gr}(r, d)$ be the Grassmannian 
	variety parameterizing quotients 
	$\mathbb{C}^r \twoheadrightarrow Q$ with $\dim Q=d$. 
	In~\cite{Kapranov}, Kapranov proved the existence of 
	a full strong exceptional collection 
	\begin{align}\label{intro:Grass}
		D^b(\mathrm{Gr}(r, d)) =\langle E_{\alpha} : \alpha \in \mathbb{B}(r, d) \rangle
		\end{align}
	where $\mathbb{B}(r, d)$ is the set of Young diagrams
	with $\mathrm{width} \le r-d$ and $\mathrm{height} \le d$.  
	
	The Grassmannian $\mathrm{Gr}(r, d)$ is regarded as a Grothendieck 
	Quot scheme over the point. By replacing the point by a curve, we 
	obtain the Quot scheme of points over a curve. 
	Let $C$ be a smooth projective curve over $\mathbb{C}$, 
	and $\eE \to C$ a vector bundle on it of rank $r$. 
	We consider the Quot scheme 
	parameterizing zero-dimensional quotients of $\eE$ with length $d$
	\begin{align*}
		\mathrm{Quot}_C(\eE, d)=\{ \eE \twoheadrightarrow \qQ : 
		\dim \Supp(\qQ)=0, 
		\mathrm{length}(\qQ)=d\}. 
		\end{align*}
	The Quot scheme $\Quot_{C}(\eE, d)$ is a smooth projective variety of dimension $rd$. 
	When $r=1$, it is isomorphic to the symmetric product 
	$\mathrm{Sym}^d(C)$. 
	In this paper, we prove the following structure of the derived category 
	of $\Quot_C(\eE, d)$, which gives a one dimensional analogue of (\ref{intro:Grass}): 
	\begin{thm}\label{intro:thm1}
		There is a semiorthogonal decomposition of the form 
		\begin{align}\label{intro:sodQ}
			D^b(\Quot_C(\eE, d))=\left\langle D^b(\mathrm{Sym}^{d_1}(C) \times 
			\cdots \times \mathrm{Sym}^{d_r}(C)) : 
			d_1+\cdots+ d_r=d  \right\rangle. 
			\end{align}
		Here the order of the above semiorthogonal summands is given by a
		lexicographic order of $(d_1, \ldots, d_r) \in \mathbb{Z}_{\ge 0}^r$. 
		\end{thm}
	In~\cite[Proposition~4.5]{BFP}, Bagnarol-Fantechi-Perroni proved the following 
	identity of 
	the class of $\Quot_C(\eE, d)$ in the Grothendieck ring of varieties (also see~\cite{Bifet}): 
	\begin{align*}
		[\Quot_C(\eE, d)]=\sum_{d_1+\cdots+d_r=d}
		[\mathrm{Sym}^{d_1}(C)]\times \cdots \times[\mathrm{Sym}^{d_r}(C)] \times [\mathbb{A}^1]^{l_r}
		\end{align*}
	where $l_r \cneq \sum_{i=1}^r (i-1)d_i$. 
	The semiorthogonal decomposition in Theorem~\ref{intro:thm1} gives a categorical 
	analogue of the above identity. 
	We also refer to~\cite{OpPan} on enumerative geometry 
	related to $\Quot_{C}(\eE, d)$. 
We also note that each semiorthogonal summand
 (\ref{intro:sodQ}) may be further decomposed (see Remark~\ref{rmk:sym}). 
	
	\subsection{Quot formula of relative dimension one}
	The result of Theorem~\ref{intro:thm1} is a special case of 
	the Quot formula of relative dimension one, which is described below. 
	Let $S$ be a smooth quasi-projective scheme 
	and 
	\begin{align*}
		\pi \colon \cC \to S
		\end{align*}
	 be a smooth projective morphism of relative dimension one. 
For $\eE \in \Coh(\cC)$ with rank $r$, we consider the $S$-relative Quot scheme 
\begin{align}\label{intro:quot}
	\Quot_{\cC/S}(\eE, d) \to S
	\end{align}
whose fiber over $s \in S$ 
is the Quot scheme $\Quot_{\cC_s}(\eE_s, d)$
over the curve $\cC_s$. As we do not require $\eE$ to be locally free nor $S$-flat, 
several geometric properties of $\Quot_{\cC_s}(\eE_s, d)$ (e.g. dimension, smoothness)
may depend on $s \in S$. 

Let $\hH \cneq \eE xt^1_{\cC}(\eE, \oO_{\cC})$. 
As a dual side of (\ref{intro:quot}), 
we also consider the $S$-relative Quot scheme 
$\Quot_{\cC/S}(\hH, d) \to S$. 
If we furthermore assume that $\eE$ has 
homological dimension less than or equal to one, 
there exist quasi-smooth derived schemes 
\begin{align*}
	\mathbf{Quot}_{\cC/S}(\eE, d) \to S \leftarrow 
	\mathbf{Quot}_{\cC/S}(\hH, d)
	\end{align*}
whose classical truncations are
$\Quot_{\cC/S}(\eE, d)$, $\Quot_{\cC/S}(\hH, d)$, 
with virtual dimensions $\dim S+rd$, $\dim S-rd$, 
respectively (if non-empty). 
We also denote by 
\begin{align*}
	\mathrm{Sym}^k(\cC/S) \cneq 
	\overbrace{(\cC \times_S \cdots \times_S \cC)}^k /\mathfrak{S}_k \to S
	\end{align*}
the relative symmetric product, whose fiber at $s \in S$ is 
$\mathrm{Sym}^k(\cC_s)$. 
We prove the following: 
\begin{thm}\label{intro:thm2}
	There is a semiorthogonal decomposition of the form 
	\begin{align*}
		&D^b(\mathbf{Quot}_{\cC/S}(\eE, d))
		= \\
		&\left\langle D^b(\mathrm{Sym}^{d_1}(\cC/S) \times_S
		\cdots \times_S \mathrm{Sym}^{d_r}(\cC/S) \times_S
		\mathbf{Quot}_{\cC/S}(\hH, d-\sum_{i=1}^r d_i) ) :
		(d_1, \ldots, d_r)\in \mathbb{Z}_{\ge 0}^r \right\rangle. 
		\end{align*}
	\end{thm}
	
	When $\cC=S$, 
	the following Quot formula 
	is conjectured in~\cite{JiangQuot} and proved in~\cite{ToQuot}:
		\begin{align}\label{quot:formula}
		D^b(\mathbf{Quot}_{S}(\eE, d))=
		\left\langle \binom{r}{k}\mbox{-copies of }
		D^b(\mathbf{Quot}_{S}(\hH, d-k)) : 
		0\le k\le \mathrm{min}\{d, r\}  \right\rangle. 		
	\end{align}
The result of Theorem~\ref{intro:thm2} gives a 
relative dimension one analogue of the above 
semiorthogonal decomposition. 
Note that Theorem~\ref{intro:thm1} is obtained from Theorem~\ref{intro:thm2}
by taking $S$ to be the point and $\eE$ to be a vector bundle. 
In some sense, 
the semiorthogonal decomposition in Theorem~\ref{intro:thm2} may be 
regarded as a bosonic counterpart of the semiorthogonal decomposition (\ref{quot:formula})
(see Remark~\ref{rmk:boson}).

\subsection{Categorical wall-crossing formula for framed one loop quiver}	
The proof of Theorem~\ref{intro:thm2} is similar to the 
proof of the formula (\ref{quot:formula}) in~\cite{ToQuot}. 
Namely we construct relevant functors using global 
categorified Hall products, and reduce the problem to 
a local situation. 
In~\cite{ToQuot}, the required local statement 
is semiorthogonal decomposition under Grassmannian flip 
proved in~\cite{Toconi}. 
In the case of Theorem~\ref{intro:thm2}, 
the required local statement is 
semiorthogonal decomposition under 
wall-crossing of framed one loop quivers. 

We denote by $Q$ the quiver with one vertex $\{1\}$ and one loop. 
For $a, b \in \mathbb{Z}_{\ge 0}$ with 
$r \cneq a-b \ge 0$, 
we denote by $\qQ_{a, b}$ the quiver with 
vertices $\{\infty, 1\}$, 
where there are $a$-arrows from $\infty$ to $1$, 
$b$-arrows from $1$ to $\infty$: 
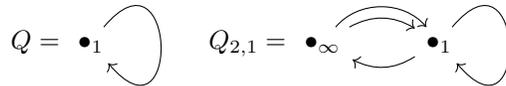
	\[
		Q=
	\begin{tikzcd}
		\bullet_{1}
		\arrow[out=50, in=310, loop]
	\end{tikzcd}   \quad  
	Q_{2, 1}=
	\begin{tikzcd}
		\bullet_{\infty}
		\arrow[r,bend left]
		\arrow[r,bend left=50] &
		\bullet_1
		\arrow[out=50, in=310, loop]
		\arrow[l, bend left]
	\end{tikzcd}
	\]
For $d \in \mathbb{Z}$, we denote by 
$\mM_{Q}(d)$ the moduli stack of $Q$-representations with 
dimension vector $d$, 
$\gG_{a, b}(d)$ the $\C$-rigidified 
moduli stack of $Q_{a, b}$-representations 
with dimension vector $(1, d)$.
There are two GIT stable loci $G_{a, b}^{\pm}(d) 
\subset \gG_{a, b}(d)$ with respect to 
the characters $\chi_0^{\pm}$, 
where $\chi_0 \colon \mathrm{GL}(d) \to \mathbb{C}^{\ast}$ is the determinant 
chraacter. 
They are 
related by a flip $(a>b)$, flop $(a=b)$
\begin{align*}
	\xymatrix{
	G_{a, b}^+(d) \ar[rd] \ar@{.>}[rr] & &  G_{a, b}^-(d) \ar[ld] \\
	& G_{a, b}(d). &
}
	\end{align*} 
Here $\gG_{a, b}(d) \to G_{a, b}(d)$ is the good moduli space. 
The D/K principle by Bondal-Orlov~\cite{B-O2} and Kawamata~\cite{Ka1}
predicts a fully-faithful functor 
$D^b(G_{a, b}^-(d)) \hookrightarrow D^b(G_{a, b}^+(d))$. 
This is not difficult to prove, but 
we need more: we need to 
describe its semiorthogonal complements
in terms of $D^b(G^-_{a, b}(d'))$ for $d'<d$.
Similarly to the argument in~\cite{Toconi}, we 
prove it using categorified Hall products
introduced in~\cite{Tudor1.5}. 
It is a functor 
\begin{align}\label{intro:catHA}
	\ast \colon 
	D^b(\mM_{Q}(d_1)) \boxtimes D^b(\gG_{a, b}(d_2)) \to 
	D^b(\gG_{a, b}(d_1+d_2)).
	\end{align}	
The above functor is defined by the stack of short exact sequences of 
$Q_{a, b}$-representations, and categorifies cohomological Hall 
algebras in~\cite{MR2851153}. 

We will use the following two subcategories for $c \in \mathbb{Z}$:
\begin{align}\label{intro:window}
	\mathbb{W}(d) \subset D^b(\mM_Q(d)), \ 
	\mathbb{W}_c(d) \subset D^b(\gG_{a, b}(d)). 
	\end{align}
The first one is the subcategory introduced in~\cite{Tudor1.7}, 
which gives an approximation of BPS sheaf in 
the context of cohomological DT theory in~\cite{DaMe} (in~\cite{PTzero}, a similar 
category for the triple loop quiver with a super-potential is called a a quasi-BPS category). 
For a general symmetric quiver it seems difficult to investigate 
it, but in our situation of the quiver $Q$ 
the structure of $\mathbb{W}(d)$ is very simple: 
it is just generated by $\oO_{\mM_Q(d)}$ and equivalent to 
$D^b(M_Q(d))$
for the good moduli space 
\begin{align*}
	\mM_Q(d) \to M_Q(d)=\mathrm{Sym}^d(\mathbb{A}^1) \cong \mathbb{A}^d.
	\end{align*}
The second one is a window subcategory 
introduced and studied in~\cite{MR3327537, MR3895631, HLKSAM}. 	
It has the property that the composition functors
\begin{align*}
	\mathbb{W}_c(d) \hookrightarrow D^b(\gG_{a, b}(d)) \twoheadrightarrow 
	D^b(G_{a, b}^{\pm}(d))
	\end{align*}
are equivalences for $(c=a, +)$, $(c=b, -)$. 
The definitions of the subcategories (\ref{intro:window}) are similar, 
but because of the generic stabilizers of $\mM_Q(d)$
their play different roles in this paper. 

We show that, for each $(d_1, \ldots, d_r) \in \mathbb{Z}_{\ge 0}^r$, 
the categorified Hall product 
induces the fully-faithful functor 
\begin{align*}
	\ast \colon 
	\mathbb{W}(d_1) \boxtimes (\mathbb{W}(d_2) \otimes \chi_0)
	\boxtimes \cdots \boxtimes (\mathbb{W}(d_l) \otimes \chi_0^{r-1})
	\boxtimes (\mathbb{W}_b(d-\sum_{i=1}^r d_i) \otimes \chi_0^r)
	\to \mathbb{W}_a(d)
\end{align*}
whose essential images form a semiorthogonal decomposition
(see Theorem~\ref{thm:SOD}). 
In particular, we have the following: 
\begin{thm}
\emph{(Corollary~\ref{cor:sod})}\label{intro:thmone}
	There is a semiorthogonal decomposition 
	of the form
\begin{align*}
	D^b(G_{a, b}^+(d))
	=\left\langle D^b(M_Q(d_1)) \boxtimes \cdots 
	\boxtimes D^b(M_Q(d_r)) \boxtimes D^b(G_{a, b}^-(d-\sum_{i=1}^r d_i)) : (d_1, \ldots, d_r) \in \mathbb{Z}_{\ge 0}^r \right\rangle. 
\end{align*}
\end{thm}

In~\cite{Toconi}, a categorical wall-crossing formula for 
framed zero-loop quivers is applied to give a categorical wall-crossing 
formula for the resolved conifold. 
In~\cite{PTzero}, a categorical wall-crossing formula 
for framed triple-loop quiver is obtained 
to give a categorical wall-crossing formula for $\mathbb{C}^3$. 
The result of Theorem~\ref{intro:thmone}
can be also used to give a categorical wall-crossing formula 
in other situations, e.g. 
a CY 3-fold which contracts a divisor to a curve.

	\subsection{Acknowledgements}
	The author is grateful to
	Tudor P\u adurariu for a joint work~\cite{PTzero}, whose subject 
	is related to this paper. 
The author is supported by World Premier International Research Center
Initiative (WPI initiative), MEXT, Japan, and Grant-in Aid for Scientific
Research grant (No.~19H01779) from MEXT, Japan.	
	
\subsection{Notation and convention}
In this paper, all the schemes and (derived) stacks are 
defined over $\mathbb{C}$. 
For a derived stack $\xX$, 
we denote by $D^b(\xX)$ the homotopy category of 
$\infty$-category of quasi-coherent sheaves on $\xX$
with coherent cohomologies. 
For a scheme $S$ and derived stacks 
$\xX_i \to S$ over $S$ for $i=1, 2$, 
we denote by $D^b(\xX_1)\boxtimes_S D^b(\xX_2)=
D^b(\xX_1 \times_S \xX_2)$. 	
For triangulated subcategories
$\cC_i \subset D^b(\xX_i)$, 
we denote by $\cC_1 \boxtimes_S \cC_2$ the 
triangulated subcategory 
of $D^b(\xX_1)\boxtimes_S D^b(\xX_2)$
split generated by $C_1 \boxtimes C_2$ for $C_i \in \cC_i$. 
For a morphism of schemes $T \to S$
and an Artin stack $\xX \to S$, 
we write $\xX_T \cneq \xX \times_S T$. 
For a perfect complex $\eE$ on $\xX$, we denote by 
$\eE_T$ its pull-back to $\xX_T$. 

Let $S=\Spec R$ for a complete local $\mathbb{C}$-algebra $R$, 
and $T_i=\Spec A_i$ for complete local $R$-algebras $A_i$ for $i=1, 2$. 
We denote by $A_1 \widehat{\otimes}_R A_2$ the complete 
tensor product, and 
write $T_1 \widehat{\times}_R T_2 \cneq \Spec (A_1 \widehat{\otimes}_R A_2)$. 
For derived stacks $\xX_i \to T_i$ over $T_i$, 
we denote by $\xX_1 \widehat{\times}_R \xX_2 \to T_1 \widehat{\times}_R T_2$ the pull-back of 
$\xX_1 \times_R \xX_2 \to T_1 \times_R T_2$
via $T_1 \widehat{\times}_R T_2 \to T_1 \times_R T_2$. 
We denote by $D^b(\xX_1)\widehat{\boxtimes}_R D^b(\xX_2)
=D^b(\xX_1 \widehat{\times}_R \xX_2)$. 
The triangulated subcategory 
$\cC_1 \widehat{\boxtimes}_R \cC_2$
in $D^b(\xX_1)\widehat{\boxtimes}_R D^b(\xX_2)$
is defined to be split
generated by $(C_1 \boxtimes C_2)|_{\xX_1 \widehat{\times}_R \xX_2 }$
for $C_i \in \cC_i$. 
These notation also apply to categories of ($\C$-equivariant) 
factorizations in 
Subsection~\ref{subsec:MF} in an obvious way. 

	Let $G$ be a reductive algebraic group with maximal torus $T$.
We denote by 
$M$ the character lattice of $T$ and $N$ the cocharacter lattice
of $T$. For a $G$-representation $Y$, 
we denote by $\wt_T(Y) \subset M$ the set of $T$-weights of $Y$. 
There is a perfect pairing 
\begin{align*}
	\langle -, -\rangle \colon N \times M \to \mathbb{Z}. 
\end{align*}	
The Weyl group of $G$ is denoted by $W$, and 
$M^W \subset M$ is defined to be the fixed part of $W$-action on $M$. 
We fix a Borel subgroup $B \subset G$ and set roots of $B$ to be negative roots. 
We denote by $M^+ \subset M$ the dominant chamber, and 
for $\chi \in M^+$ we denote by
 $V(\chi)$ the irreducible $G$-representation with 
highest weight $\chi$. We also define $\rho \in M_{\mathbb{Q}}$ to be the half sum of positive roots. 
For $\chi \in M$ and $w \in W$, define 
\begin{align*}
	w \ast \chi \cneq w(\chi+\rho)-\rho. 
\end{align*}
If $\chi+\rho$ has a trivial stabilizer in $W$, 
there is a unique $w \in W$ such that 
$w \ast \chi \in M^+$, and in that case we set 
$\chi^+ \cneq w \ast \chi$.
Otherwise we set $V(\chi^+)$ to be zero.

	\section{Preliminary}

	\subsection{Attracting loci}
Let $Y$ be a smooth affine scheme with an action of a reductive 
algebraic group $G$. 
	For a one parameter subgroup $\lambda \colon \C \to G$, 
	let $Y^{\lambda \ge 0}$, $Y^{\lambda=0}$ be defined by 
	\begin{align*}
		Y^{\lambda \ge 0} &\cneq \{y \in Y: 
		\lim_{t\to 0} \lambda(t)(y) \mbox{ exists }\}, \\
		Y^{\lambda = 0} &\cneq \{y \in Y: 
		\lambda(t)(y)=y \mbox{ for all } t \in \C\}. 
	\end{align*}
	The Levi subgroup and the parabolic subgroup
	\begin{align*}
		G^{\lambda=0} \subset G^{\lambda \ge 0} \subset G
	\end{align*}
	are also similarly defined by the conjugate $G$-action on $G$, i.e. 
	$g\cdot (-)=g(-)g^{-1}$. The $G$-action on 
	$Y$ restricts to the $G^{\lambda \ge 0}$-action 
	on $Y^{\lambda \ge 0}$, 
	and the $G^{\lambda=0}$-action on $Y^{\lambda=0}$. 
	We note that
	$\lambda$ factors through 
	$\lambda \colon \C \to G^{\lambda=0}$, and it
	acts on $Y^{\lambda=0}$ trivially. 
	So we have the decomposition into fixed 
	$\lambda$-weight subcategories
	\begin{align*}
		D^b([Y^{\lambda=0}/G^{\lambda=0}])=
		\bigoplus_{j \in \mathbb{Z}}D^b([Y^{\lambda=0}/G^{\lambda=0}])_{\lambda \mathchar`- \wt=j}.
	\end{align*}
	We have the diagram of 
	attracting loci
	\begin{align}\label{dia:attract0}
		\xymatrix{
			[Y^{\lambda \ge 0}/G^{\lambda \ge 0}] \ar[r]^-{p_{\lambda}}	 \ar[d]^-{q_{\lambda}}	&
			[Y/G] \\
			[Y^{\lambda=0}/G^{\lambda=0}]\ar@/^10pt/[u]^-{\sigma_{\lambda}}. & 	
		}
	\end{align}
	Here $p_{\lambda}$ is induced by the inclusion 
	$Y^{\lambda \ge 0} \subset Y$, and $q_{\lambda}$ is given by 
	taking the $t \to 0$ limit of the action of $\lambda(t)$
	for $t \in \C$. The morphism $\sigma_{\lambda}$ is a section of 
	$q_{\lambda}$ induced by inclusions $Y^{\lambda=0} \subset Y^{\lambda \ge 0}$
	and $G^{\lambda=0} \subset G^{\lambda \ge 0}$. 
	We will use the following lemma: 
	\begin{lem}\label{lem:vanish}\emph{(\cite[Corollary~3.17, Amplification~3.18]{MR3327537})}
		
		(i)
		For $\eE_i \in D^b([Y^{\lambda \ge 0}/G^{\lambda \ge 0}])$
		with $i=1, 2$, suppose that 
		\begin{align*}
			\sigma_{\lambda}^{\ast}\eE_1 \in D^b([Y^{\lambda=0}/G^{\lambda=0}])_{\lambda \mathchar`- \wt\ge j}, \ \sigma_{\lambda}^{\ast}\eE_2 \in D^b([Y^{\lambda=0}/G^{\lambda=0}])_{\lambda \mathchar`- \wt <j}
		\end{align*}
		for some $j$. 
		Then $\Hom(\eE_1, \eE_2)=0$. 
		
		(ii) For $j \in \mathbb{Z}$, the functor 
		\begin{align*}
			q_{\lambda}^{\ast} \colon 
			D^b([Y^{\lambda=0}/G^{\lambda=0}])_{\lambda  \mathchar`- \wt=j}
			\to D^b([Y^{\lambda \ge 0}/G^{\lambda \ge 0}])
		\end{align*} 
		is fully-faithful. 
		\end{lem}
		\subsection{Kempf-Ness stratification}\label{subsec:exam:KN}
	Here review Kempf-Ness stratifications associated with GIT quotients of 
	reductive algebraic groups
	following the convention of~\cite[Section~2.1]{MR3327537}. 
	Let $Y$ and $G$ be as in the previous subsection. 
	For an element 
	$l \in 
	\Pic([Y/G])_{\mathbb{R}}$, 
	we have the open subset of $l$-semistable points
	\begin{align}\notag
		Y^{l\sss} \subset Y
	\end{align}
	characterized 
	by the set of points $y \in Y$ such that 
	for any one parameter subgroup $\lambda \colon \mathbb{C}^{\ast} \to G$
	such that 
	the limit $z=\lim_{t\to 0}\lambda(t)(y)$
	exists in $Y$, we 
	have $\wt(l|_{z})\ge 0$. 
	Let $\lvert \ast \rvert$ be the Weyl-invariant norm 
	on $N_{\mathbb{R}}$. 
	The above subset of $l$-semistable 
	points fits into the \textit{Kempf-Ness (KN) stratification} 
	\begin{align}\label{KN:strata}
		Y=S_{1} \sqcup S_{2} \sqcup \cdots \sqcup S_N \sqcup 
		Y^{l\sss}.  
	\end{align}
	Here for each $1\le i\le N$ there exists a 
	one parameter subgroup $\lambda_{i} \colon \mathbb{C}^{\ast} \to
	T \subset G$, an open and closed subset 
	$Z_{i}$ of
	$(Y \setminus \cup_{i'<i} S_{i'})^{\lambda_{i}=0}$
	(called \textit{center} of $S_i$)
	such that 
	\begin{align*}
		S_{i}=G \cdot Y_{i}, \ 
		Y_{i}\cneq \{ y \in Y^{\lambda_{i} \ge 0}: 
		\lim_{t \to 0}\lambda_{i}(t)(y) \in Z_{i}\}. 
	\end{align*}
	Moreover by setting the slope to be
	\begin{align}\label{slope}
		\mu_{i} \cneq -\frac{
			\wt(l|_{Z_{i}})}{\lvert \lambda_{i} \rvert} \in \mathbb{R}
	\end{align}
	we have 
	the inequalities
	$\mu_1>\mu_2>\cdots>0$.
	We have the following diagram (see~\cite[Definition~2.2]{MR3327537})
	\begin{align}\label{dia:YZ}
		\xymatrix{
			[Y_{i}/G^{\lambda_{i}\ge 0}] \ar[r]^-{\cong} \ar[d] & [S_{i}/G] \ar[dl]_{p_{i}} \ar@<-0.3ex>@{^{(}->}[r]^-{q_{i}}
			& \left[\left(Y\setminus \cup_{i'<i} S_{i'}\right)/G \right]  \\
			[Z_{i}/G^{\lambda_{i}=0}]. \ar@<-0.3ex>@{^{(}->}[rru]_-{\tau_{i}} & & 
		}
	\end{align}
	Here the left vertical arrow is given by
	taking the $t\to 0$ limit of the action of $\lambda_{i}(t)$
	for $t \in \C$, and $\tau_{i}, q_{i}$ are induced by the 
	embedding $Z_{i} \hookrightarrow Y$, $S_{i} \hookrightarrow Y$ respectively. 
	
	Halpern-Leistner~\cite{Halpinstab} extended 
	the above notion of Kempf-Ness stratifications to $\Theta$-stratifications 
	for more general Artin stacks (see~\cite[Definition~2.2]{Halpinstab}). 
	Let $\nN$ be a classical Artin stack 
	locally of finite type and with affine stabilizers. 
	Suppose that it admits a good moduli space $\nN \to N$ (see~\cite{MR3237451}). 
	Then for any $l \in \Pic(\nN)_{\mathbb{Q}}$
	and a positive definite $b \in H^4(\nN, \mathbb{Q})$
	(which corresponds to Weyl-invariant norm on $N_{\mathbb{R}}$ in the above 
	setting), 
	there is an associated $\Theta$-stratification 
	(see~\cite[Theorem~4.1.3]{HalpK3})
	\begin{align*}
	\nN=\sS_1 \sqcup \sS_2 \sqcup \cdots \sqcup \sS_N \sqcup \nN^{l\sss}. 	
		\end{align*}
	Here $b$ is called positive definite if for any non-degenerate $f \colon B\C \to \nN$, 
	we have $q^{-2} f^{\ast}b>0$, 
	where $q$ is the generator of $H^{\ast}(B\C)=\mathbb{Q}[q]$. 
	Similarly to KN stratifications, there
	are associated closed substacks $\zZ_i \subset \sS_i$
	(called center of $\sS_i)$
	with canonical $\C$-stabilizers at each point of $\zZ_i$, 
	 and a diagram of attracting loci similar to (\ref{dia:YZ})
	\begin{align}\label{dia:NS}
		\xymatrix{ \sS_i \ar@<-0.3ex>@{^{(}->}[r] \ar[d] & 
			\nN \setminus \bigcup_{i'<i} \sS_{i'} \\
			\zZ_i \ar@<-0.3ex>@{^{(}->}[ru]_-{\tau_{i}}. &
			}
		\end{align}
	
	\subsection{Window theorem}
In the diagram (\ref{dia:YZ}), 
let $\eta_{i} \in \mathbb{Z}$ be defined by 
	\begin{align}\label{etai}
		\eta_{i} \cneq \wt_{\lambda_{i}}(\det(N_{S_{i}/Y}^{\vee}|_{Z_{i}})).
	\end{align}
	In the case that $Y$ is a $G$-representation, 
	it is also written as 
	\begin{align*}
		\eta_i =\langle \lambda_i, (Y^{\vee})^{\lambda_i>0}-(\mathfrak{g}^{\vee})^{\lambda_i>0}   \rangle. 
	\end{align*}
	Here for a $G$-representation $W$ and a one parameter 
	subgroup $\lambda \colon \C \to T$, 
	we denote by $W^{\lambda>0} \in K(BT)$ the 
	subspace of $W$ spanned by weights which pair positively with $\lambda$. 
	We will use the following version of window theorem:
	\begin{thm}\label{thm:window}\emph{(\cite{MR3327537, MR3895631})}
		For each $i$, we take $m_{i} \in \mathbb{R}$. 
		Let
		\begin{align}\label{window:m}
			\mathbb{W}_{m_{\bullet}}^l([Y/G]) \subset 
			D^b([Y/G])
		\end{align}
		be the subcategory of objects $\pP$
		satisfying the condition
		\begin{align}\label{condition:P}
			\tau_{i}^{\ast}(\pP) \in 
			\bigoplus_{j \in [m_{i}, m_{i}+\eta_{i})}
			D^b([Z_{i}/G^{\lambda_i=0}])_{\lambda_{i} \mathchar`- \wt= j}
		\end{align}
		for all $1\le i\le N$. 
		Then the composition functor 
		\begin{align*}
			\mathbb{W}_{m_{\bullet}}^l([Y/G]) \hookrightarrow 
			D^b([Y/G])
			\twoheadrightarrow
			D^b([Y^{l\sss}/G])
		\end{align*}
		is an equivalence. 
	\end{thm}

Suppose that $Y$ is a symmetric $G$-representation, i.e. 
$Y \cong Y^{\vee}$ as $G$-representations. 
In this case, there 
is another version of window theorem~\cite{HLKSAM}, called \textit{magic window theorem}. 
We denote by $\Sigma \subset M_{\mathbb{R}}$ the subset
\begin{align}\label{sub:W}
	\Sigma=\sum_{\gamma \in \mathrm{wt}_T(Y)} [0, 1] \cdot \gamma \subset M_{\mathbb{R}}. 
	\end{align}
Namely $\Sigma$ is the convex hull of $T$-weights of $\bigwedge^{\ast}(Y)$. 
For $\delta \in M_{\mathbb{R}}^{W}$, the magic window subcategory 
\begin{align}\label{magic}
	\mathbb{W}_{\delta} \subset D^b([Y/G])
	\end{align}
is defined to be split generated by 
$V(\chi) \otimes \oO_Y$
for $\chi \in M^+$ satisfying 
\begin{align*}
	\chi+\rho \in \frac{1}{2}\Sigma +\delta. 
	\end{align*}
An element in $M_{\mathbb{R}}^{W}$ is called $\Sigma$-\textit{generic}
 if it lies in the 
linear span of $\Sigma$ but is not parallel to any face of $\Sigma$. 
\begin{thm}\emph{(\cite[Theorem~3.2]{HLKSAM})}
	We take $\delta, l \in M_{\mathbb{R}}^W$ such that 
	$\partial(\Sigma/2 +\delta) \cap M^+=\emptyset$. 
	Then the composition functor
	\begin{align*}
		\mathbb{W}_{\delta} \hookrightarrow D^b([Y/G]) \twoheadrightarrow 
		D^b([Y^{l\sss}/G])
		\end{align*}
	is fully-faithful. 
	If there is a $\Sigma$-generic element in $M_{\mathbb{R}}^W$, 
	then the above functor is an equivalence 
	whenever $Y^{l\sss}=Y^{l\st}$, where $Y^{l\st}$ is the $l$-stable part.  
	\end{thm}

\subsection{Categorified Hall product}
In the setting of the diagram (\ref{dia:attract0}), 
since $p_{\lambda}$ is proper
we have the functor
\begin{align}\label{cat:prod}
	p_{\lambda \ast} q_{\lambda}^{\ast} \colon 
	D^b([Y^{\lambda=0}/G^{\lambda=0}]) \to 
	D^b([Y^{\lambda \ge 0}/G^{\lambda \ge 0}])
	\to 	D^b([Y/G]). 
	\end{align}
In the case that $Y$ is a moduli stack of representations of quivers, 
the above functor gives a categorified Hall product~\cite{Tudor1.5}. 
The image of $V(\chi) \otimes \oO_{Y^{\lambda=0}}$ under the 
above functor is calculated using 
Borel-Weil-Bott theorem. 
We will use the following fact: 
\begin{prop}\emph{(\cite[Proposition~3.8]{HLKSAM})}\label{prop:resol}
	For $\chi \in M^+$, 
	the object 	
	\begin{align*}
		p_{\lambda \ast} q_{\lambda}^{\ast}(V(\chi) \otimes \oO_{Y^{\lambda=0}})
		 \in D^b([Y/G])
		 \end{align*}
	is a successive extension of objects of the form 
	$V((\chi-\sigma_I)^+) \otimes \oO_{Y}[\sharp I -l(I)]$. 
		Here $I$ is a finite subset
	\begin{align*}
		I \subset \{ \beta \in \wt_{T}(Y) : \langle \beta, \lambda \rangle 
		<0\}, 
		\end{align*}
	$\sigma_I$ is the sum of $\beta \in I$
	and $l(I)$ is the length of 
	$w \in W$ with $w \ast (\chi -\sigma_I) \in M^+$. 
		Moreover if $(\chi-\sigma_I)^+=\chi$ implies $\sharp I=0$, 
	then $V(\chi) \otimes \oO_Y$ appears exactly once. 
	\end{prop}

		\subsection{The category of factorizations}\label{subsec:MF}
	Let $\yY$ be a smooth noetherian algebraic stack 
	over $\mathbb{C}$
	and take $w \in \Gamma(\oO_{\yY})$. 
	A (coherent) factorization of $w$ consists of 
	\begin{align}\label{fact:P}
		\xymatrix{
			\pP_0 \ar@/^8pt/[r]^{\alpha_0} &  \ar@/^8pt/[l]^{\alpha_1} \pP_1,  
		} \
		\alpha_0 \circ \alpha_1=\cdot w, \ 
		\alpha_1 \circ \alpha_0=\cdot w, 
	\end{align}
	where each $\pP_i$ is a coherent sheaf on $\yY$ and 
	$\alpha_i$ is a morphism of coherent sheaves. 
	The category of coherent factorizations naturally
	forms a dg-category, whose homotopy 
	category is denoted by $\mathrm{HMF}(\yY, w)$. 
	The subcategory of absolutely acyclic 
	objects 
	\begin{align*}
		\mathrm{Acy}^{\rm{abs}} \subset \mathrm{HMF}(\yY, w)
	\end{align*}
	is defined to be the minimum thick triangulated subcategory 
	which contains totalizations of short exact sequences of 
	coherent factorizations of $w$. 
	The triangulated category of factorizations of $w$ is defined by 
	(cf.~\cite{Ornonaff, MR3366002, MR3112502})
	\begin{align*}
		\mathrm{MF}(\yY, w) \cneq \mathrm{HMF}(\yY, w)/\mathrm{Acy}^{\rm{abs}}. 
	\end{align*}
	If $\yY$ is an affine scheme, then $\MF(\yY, w)$ 
	is equivalent to Orlov's triangulated category of matrix factorizations 
	of $w$~\cite{Orsin}.
	
	Let $\yY=[Y/G]$ for a smooth quasi-projective scheme $Y$
	and $G$ is an affine algebraic group acting on $Y$. 
	Suppose that there is an auxiliary $\C$-action on $Y$ 
	which commutes with the $G$-action, 
	and the regular function $w \in \Gamma(\oO_{\yY})$ is 
	of $\C$-weight one. 
	A $\C$-equivariant (coherent) factorization of $w$ consists of 
(\ref{fact:P}) such that 
$\alpha_0$ is of $\C$-weight zero and $\alpha_1$ is of $\C$-weight one. 
The triangulated category of $\C$-equivariant factorizations of $w$ 
is also similarly defined, and denoted by 
\begin{align*}
	\MF^{\C}(\yY, w). 
	\end{align*}
If $\C$ acts on $Y$ trivially and $w=0$, 
then $\MF^{\C}(\yY, 0)$ is equivalent to $D^b(\yY)$. 

	\subsection{Koszul duality}\label{subsec:Koszul}
	For a derived Artin stack $\fM$, its $(-1)$-shifted cotangent 
	is defined by 
	\begin{align*}
		\Omega_{\fM}[-1] \cneq \Spec \Sym(\mathbb{T}_{\fM}[1]). 
	\end{align*}
	Here $\mathbb{T}_{\fM}$ is the tangent complex of $\fM$. 
	In the case that $\fM$ is a derived complete intersection, the 
	classical truncation of $\Omega_{\fM}[-1]$ has 
	the following critical locus description. 
	Let $\yY=[Y/G]$ for a smooth quasi-projective scheme $Y$
	and $G$ is an affine algebraic group acting on $Y$. 
	Let $\fF \to \yY$ be a vector bundle on it with a section $s$. 
	Suppose that $\fM$ is a derived zero locus of $s$. 
	Let $w$ be the function
	\begin{align*}
		w \colon \fF^{\vee} \to \mathbb{A}^1, \ 
		w(y, v)=\langle s(y), v \rangle
	\end{align*}
	for $y \in \yY$ and $v \in \fF^{\vee}|_{y}$. 
	Then $t_0(\Omega_{\fM}[-1])$ is 
	isomorphic to the critical locus 
	$\Crit(w)$ (see~\cite[Proposition~2.8]{MR3607000}, \cite[Lemma~2.5]{ToQuot}). 
	The above construction is summarized in the following diagram 
	\begin{align}\label{dia:-1shift}
		\xymatrix{
			\yY \ar@<-0.3ex>@{^{(}->}[r]^-{0} \diasquare& \fF \ar[d] \\
			\fM \ar@<-0.3ex>@{^{(}->}[u] \ar@<-0.3ex>@{^{(}->}[r]
			& \yY,  \ar@/_10pt/[u]_-{s} 
		}	
		\quad 
		\xymatrix{t_0(\Omega_{\fM}[-1]) \ar[r]^-{\cong} \ar[d] &
			\Crit(w) \ar@<-0.3ex>@{^{(}->}[r] \ar[d] & \fF^{\vee} \ar[d] \ar[rd]^-{w} & \\
			\fM \ar@{=}[r]& \fM \ar@<-0.3ex>@{^{(}->}[r] & \yY & \mathbb{A}^1. 
		}
	\end{align}
Let $\C$ acts on fibers of $\fF^{\vee} \to Y$ with weight one. 
In the above setting, the Koszul duality equivalence in~\cite[Proposition~4.8]{MR3631231}
(also see~\cite{MR3071664, MR2982435, TocatDT})
is the following: 
\begin{thm}\emph{(\cite{MR3631231, MR3071664, MR2982435, TocatDT})}\label{thm:koszul}
There is an equivalence
\begin{align*}
	D^b(\fM) \stackrel{\sim}{\to} \MF^{\C}(\fF^{\vee}, w). 
\end{align*}
\end{thm}

Let $\psi \colon F_0 \to F_1$ be a 
morphism of $G$-equivariant vector bundles on $Y$
and set $\fF_i=[F_i/G]$. 
We denote by 
\begin{align*}
	(\fF_0 \stackrel{\psi}{\to} \fF_1) \in D^b(\yY)
	\end{align*}
the associated two term complex. 
We set $F^{-i}=F_i^{\vee}$, $\fF^{-i}=\fF_i^{\vee}$, and 
consider the following total space
\begin{align*}
	\fF_0 \times_{\yY} \fF^{-1}=\left[(F_0 \times_Y F^{-1})/G\right]. 
	\end{align*}
There is a regular function $w$ on it
\begin{align}\label{funct:Fw}
	w \colon \fF_0 \times_{\yY} \fF^{-1} \to \mathbb{A}^1, \ 
	(y, u, v) \mapsto \langle \psi|_{y}(u), v\rangle =\langle 
	u, \psi|_{y}^{\vee}(v) \rangle. 
	\end{align}
Here $y \in Y$, $u \in F_0|_{y}$, 
$v \in F^{-1}|_{y}$, and 
$\psi^{\vee} \colon F^{-1} \to F^0$ is the dual of $\psi$. 
We have two auxiliary $\C$-actions on $F_0 \times_Y F^{-1}$ which 
commute with the $G$-action: 
the one is acting on the fibers of $F^{-1} \to Y$ by weight one, 
and the another is acting on the fibers of $F_0 \to Y$ by weight one. 
The function $w$ is of weight one with respect to both 
of the above actions, so we 
obtain two $\C$-equivariant categories of factorizations 
\begin{align*}
	\MF^{\C}(\fF_0 \times_{\yY} \fF^{-1}, w), \ 
	\MF^{\C}(\fF_0 \times_{\yY} \fF^{-1}, w)'
\end{align*}
where the left one is defined by 
the $\C$-action on $F^{-1}$, and the right one is defined 
by the $\C$-action on $F_0$.

Suppose that there is a one dimensional subtorus $\C \subset G$
which lies in the center of $G$ and acts on $Y$ trivially. 
Then $\C \subset G$ acts on fibers of $F_i \to Y$, and we assume that 
they are of weight one. 
In this situation, the following lemma is implicit in~\cite[Subsection~2.2]{koTo}. 
\begin{lem}\label{lem:action}
	In the above situation, there is an equivalence 
	\begin{align}\label{MF:eq}
		\MF^{\C}(\fF_0 \times_{\yY} \fF^{-1}, w) \simeq
	\MF^{\C}(\fF_0 \times_{\yY} \fF^{-1}, w)'.
	\end{align}
	\end{lem}
\begin{proof}
	The isomorphism of algebraic groups 
	\begin{align}\label{isom:group}
		G \times \C \stackrel{\cong}{\to} G \times \C, \ 
		(g, t) \mapsto (t^{-1}g, t)
		\end{align}
	gives an isomorphism of stacks 
	\begin{align}\label{isom:stack}
	\left[(F_0 \times_Y F^{-1})/(G \times \C)\right] \stackrel{\cong}{\to}
		\left[(F_0 \times_Y F^{-1})/'(G \times \C)\right].
		\end{align}
	Here in the left hand side 
	the second $\C$ acts on the fibers of $F^{-1} \to Y$ by weight one, 
	and in the right hand side 
	it acts on the fibers of $F_0 \to Y$ by weight one. 
	Since (\ref{isom:group}) commutes with the second projection, the isomorphism (\ref{isom:stack})
	induces the equivalence (\ref{MF:eq}). 
	\end{proof}

	\begin{rmk}\label{rmk:linear}
		Let $\fU$, $\fU^{\vee}$ be the derived zero loci
		\begin{align*}
			\fU=(\psi=0) \subset \fF_0, \ 
			\fU^{\vee}=(\psi^{\vee}=0) \subset \fF^{-1}. 
			\end{align*}
		Then Theorem~\ref{thm:koszul} together with 
		Lemma~\ref{lem:action}
		implies an equivalence
		$D^b(\fU) \simeq D^b(\fU^{\vee})$,  
		which recovers linear Koszul duality in~\cite{MrRi}. 
		\end{rmk}
	\section{Categorical wall-crossing for framed 
	one loop quiver}
\subsection{One loop quiver}
We denote by $Q$ the one loop quiver, i.e. 
it consists of one vertex $\{1\}$ and one loop:
\begin{align}\label{loop:Q}
	Q=
 \begin{tikzcd}
	\bullet_{1}
	\arrow[out=50, in=310, loop]
\end{tikzcd}
\end{align}
For $d \in \mathbb{Z}_{\ge 0}$, let $V$ be a $d$-dimensional vector 
space. 
The quotient stack 
\begin{align*}
	\mM_{Q}(d) \cneq [\End(V)/\GL(V)]
	\end{align*}
is the moduli stack of $Q$-representations of dimension $d$. 
We have the good moduli space
\begin{align}\label{mor:MQ}
\pi_{Q} \colon 	\mM_Q(d) \to M_Q(d) \cneq \End(V)\ssslash \GL(V)
\stackrel{\cong}{\to} \mathbb{A}^d. 
	\end{align}
Here the last isomorphism is given by 
assigning $A \in \End(V)$ to the coefficients of 
its characteristic polynomial. 
\begin{rmk}\label{rmk:MQ}
The stack $\mM_Q(d)$ is isomorphic to the stack of 
zero-dimensional sheaves $Q$ on $\mathbb{A}^1$ with $\chi(Q)=d$. 
The morphism (\ref{mor:MQ}) is identified with 
the Hilbert-Chow morphism 
sending $Q$ to the support of $Q$
in $\mathrm{Sym}^d(\mathbb{A}^1) \cong \mathbb{A}^d$. 
\end{rmk}

We fix a basis of $V$, 
and a maximal torus $T \subset \GL(V)$
to be consisting of diagonal matrices. 
We also set a Borel subgroup $B \subset \GL(V)$ 
to be consisting of upper triangular matrices, where roots of $B$ 
are set to be negative roots. 
The character lattice $M$ for $T$
is given by $M=\mathbb{Z}^d$, 
and the dominant chamber $M^+ \subset M$ is given by 
\begin{align*}
	M_{\mathbb{R}}^+=\{(x_1, x_2, \ldots, x_d) \in \mathbb{R}^d : 
	x_1 \le x_2 \le \cdots \le x_d\}. 
\end{align*}
We also often denote the standard basis of $M$ 
as $\{e_1, \ldots, e_d\}$ and write
an element of $M_{\mathbb{R}}$ as 
$x_1 e_1+ \cdots + x_d e_d$. 
The half sum of positive roots $\rho$ is given by 
\begin{align*}
	\rho=\frac{1}{2} \sum_{i>j}(e_i-e_j)
	=\left(-\frac{d-1}{2}, -\frac{d-3}{2}, \ldots, 
	\frac{d-1}{2}   \right). 
\end{align*}
The Weyl group of $\GL(V)$ is the symmetric group $S_d$, and 
the Weyl-invariant part $M^W$ is generated by $\chi_0=(1, \ldots, 1)$, 
where $\chi_0$ corresponds to the character 
\begin{align}\label{ch:det}
	\chi_0 \colon \GL(V) \to \C, \ g \mapsto \det(g). 
\end{align}

Let $\Sigma(d) \subset M_{\mathbb{R}}$ be 
the subset in (\ref{sub:W})
for the symmetric $\GL(V)$-representation $\End(V)$. 
Explicitly it is 
\begin{align*}
	\Sigma(d)=\sum_{-1\le c_{ij} \le 1, i>j} c_{ij} \cdot (e_i-e_j). 
	\end{align*} 
For $\delta=t\chi_0 \in M_{\mathbb{R}}^W$ with $t \in \mathbb{R}$, 
the subcategory 
\begin{align}\label{W(d)}
	\mathbb{W}_{\delta}(d) \subset D^b(\mM_{Q}(d))
	\end{align}
is defined 
as in (\ref{magic}), i.e. it is split generated by $V(\chi) \otimes \oO_{\mM_{Q}(d)}$, 
where $\chi \in M^+$ satisfies
$\chi+\rho \in \Sigma(d)/2 +\delta$. 
\begin{lem}\label{lem:genO}
The triangulated subcategory (\ref{W(d)}) is non-zero if and only 
if $\delta=t\chi_0$ for $t \in \mathbb{Z}$. 
In this case, 
it is 
split generated by 
$\chi_0^{\otimes t}$. 
Here we have regarded $\chi_0$ as a line bundle on $\mM_Q(d)$. 
\end{lem}
\begin{proof}
	Note that we have 
	\begin{align*}
	\frac{1}{2}\Sigma(d)-\rho+t\chi_0=
	\sum_{-1\le c_{ij} \le 0, i>j} c_{ij} \cdot (e_i-e_j) +\sum_{i=1}^d t \cdot e_i. 
		\end{align*}
	Therefore any element $\chi$ in the LHS
is written as 
\begin{align*}
	\chi=(-c_{21}-\cdots -c_{d1}+t)e_1+
	(c_{21}-c_{32}-\cdots -c_{d2}+t)e_2+\cdots 
	+(c_{d1}+\cdots +c_{d d-1}+t)e_d. 
	\end{align*}
If $\chi$ lies in the dominant chamber, we have 
$-c_{21}-\cdots -c_{d1} \le c_{d1}+\cdots +c_{d d-1}$, 
hence 
$c_{21}=\cdots =c_{d1}=c_{d2}=\cdots =c_{d d-1}=0$
as $c_{ij}\le 0$ for all $i, j$. 
By applying the same argument for other 
coefficients of $\chi$, 
we conclude that $c_{ij}=0$ for all $i, j$. 
Therefore we have 
\begin{align*}
	\left(\frac{1}{2}\Sigma(d)-\rho+t\chi_0 \right) \cap M_{\mathbb{R}}^+
	=t \chi_0,
	\end{align*}
and the lemma holds. 
	\end{proof}
Below we set 
\begin{align*}
	\mathbb{W}(d) \cneq \mathbb{W}_{\delta=0}(d) \subset D^b(\mM_Q(d)), 
	\end{align*}
which is generated by $\oO_{\mM_Q(d)}$ by Lemma~\ref{lem:genO}. 
Note that for $t\in \mathbb{Z}$, 
we have $\mathbb{W}_{\delta=t\chi_0}(d)=\mathbb{W}(d) \otimes 
\chi_0^{\otimes t}$. 
\begin{lem}\label{lem:pull}
	The pull-back by the morphism (\ref{mor:MQ}) induces the equivalence 
	\begin{align}\label{piast}
		\pi_Q^{\ast} \colon 
		D^b(M_Q(d)) \stackrel{\sim}{\to}
		\mathbb{W}(d). 
		\end{align}
	\end{lem}
\begin{proof}
	As $\pi_Q$ is 
	a good moduli space morphism, 
	it induces the isomorphism 
	\begin{align*}
		\pi_Q^{\ast} \colon \Hom_{M_Q(k)}(\oO_{M_Q(k)}, \oO_{M_Q(k)})
		\stackrel{\cong}{\to} \Hom_{\mM_Q(k)}(\oO_{\mM_Q(k)}, \oO_{\mM_{Q}(k)}). 
		\end{align*}
		Since the triangulated category 
	$D^b(M_Q(d))$ is generated by $\oO_{M_Q(d)}$, it follows that 
	the functor (\ref{piast}) is fully-faithful. 
	By Lemma~\ref{lem:genO}, 
	the functor (\ref{piast}) is also essentially surjective. 	
	\end{proof}

For a one parameter subgroup $\lambda \colon \C \to \GL(V)$, 
we have the diagram of attracting loci
\begin{align}\label{dia:attract}
	\xymatrix{
		\mM_Q(d)^{\lambda \ge 0}
		 \ar[r]^-{p_{\lambda}}	 \ar[d]^-{q_{\lambda}}	&
		\mM_Q(d) \\
	\mM_Q(d)^{\lambda=0}. & 	
	}
\end{align}
For $d=d_1+d_2$, let $\lambda \colon \C \to T \subset \GL(V)$ be given by 
\begin{align}\label{set:lambda}
	\lambda(t)=(\overbrace{t, \ldots, t}^{d_1}, \overbrace{1, \ldots, 1}^{d_2}). 
	\end{align}
Then we have 
\begin{align*}
	\mM_Q(d)^{\lambda =0}=\mM_Q(d_1) \times \mM_Q(d_2)
	\end{align*}
and the functor (\ref{cat:prod}) gives the associative categorified Hall 
product 
\begin{align}\label{ast:1}
\ast=p_{\lambda \ast}q_{\lambda}^{\ast} \colon 
D^b(\mM_Q(d_1)) \boxtimes D^b(\mM_Q(d_2)) \to D^b(\mM_Q(d)). 	
	\end{align}
\begin{rmk}
	Using Proposition~\ref{prop:resol}, 
	one can check that the 
	functor (\ref{ast:1}) induces the functor 
	\begin{align*}
		\ast \colon 
		\overbrace{\mathbb{W}(1) \boxtimes \cdots \boxtimes \mathbb{W}(1)}^d
		\to \mathbb{W}(d)
	\end{align*}
	such that $\mathbb{W}(d)$ is split 
	generated by the image of the above functor. 
	However the above functor is not fully-faithful. 
	\end{rmk}

\subsection{Moduli stacks of representations of framed one loop quiver}
Let us take 
\begin{align*}
	(a, b) \in \mathbb{Z}_{\ge 0}^2, \ r \cneq a-b \ge 0. 
\end{align*}
We denote by $Q_{a, b}$ the extended quiver of $Q$, 
that is the vertex $\{\infty\}$, the $a$-arrows from 
$\infty$ to $1$ and the $b$-arrows from $1$ to $\infty$
are added: 
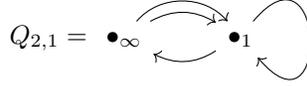
\[
Q_{2, 1}=
\begin{tikzcd}
	\bullet_{\infty}
	\arrow[r,bend left]
	\arrow[r,bend left=50] &
	\bullet_1
	\arrow[out=50, in=310, loop]
	\arrow[l, bend left]
\end{tikzcd}
\]
Let $A, B, V$ be vector spaces
such that 
\begin{align*}
	\dim A=a, \ \dim B=b, \ \dim V=d. 
	\end{align*} 
We set the following $\GL(V)$-representation 
\begin{align}\label{Yab}
	Y_{a, b}(d) \cneq \Hom(A, V) \oplus \Hom(V, B) \oplus \End(V),
	\end{align}
and form the following quotient stack 
\begin{align*}
	\gG_{a, b}(d) \cneq 
	\left[\left(\Hom(A, V) \oplus \Hom(V, B) \oplus \End(V) \right)/\GL(V)  \right].
	\end{align*}	
The above stack is the $\C$-rigidified moduli stack of $Q_{a, b}$-representations 
with dimension vector $(1, d)$. 
For a one parameter subgroup $\lambda \colon \C \to T$, 
we use the following notation for the diagram of 
attracting loci 
\begin{align}\label{dia:quiver2}
	\xymatrix{
		\gG_{a, b}(d)^{\lambda \ge 0} \ar[r]^-{p_{\lambda}} \ar[d]_-{q_{\lambda}} & \gG_{a, b}(d) \\
		\gG_{a, b}(d)^{\lambda=0}. & 	
	}
\end{align}

There exist two GIT quotients 
with respect to $\chi_0^{\pm 1}$ 
given by open substacks
\begin{align*}
	G_{a, b}^{\pm}(d) \subset \gG_{a, b}(d),
\end{align*}
which are smooth quasi-projective varieties. 
Here $\chi_0$-semistable 
locus $G_{a, b}^{+}(d)$ 
consists of 
\begin{align*}
	(\alpha, \beta, \gamma) \in \Hom(A, V) \oplus \Hom(V, B)
\oplus \End(V)
\end{align*}
such that, by setting $V_{\gamma}$ to be the 
$\mathbb{C}[Q]$-module structure on $V$ determined by $\gamma$, 
the image of 
 $\alpha \colon A \to V$
 generates $V_{\gamma}$ as a $\mathbb{C}[Q]$-module. 
 Similarly 
 $\chi_0^{-1}$-semistable locus 
$G_{a, b}^-(d)$ consists of $(\alpha, \beta, \gamma)$
such that the image of 
$\beta^{\vee} \colon B^{\vee} \to V^{\vee}$
generates $V^{\vee}_{\gamma^{\vee}}$ as 
a $\mathbb{C}[Q]$-module (see~\cite[Lemma~5.1.9]{TocatDT}).  

Let $G_{a, b}(d)$ be the good moduli space
for $\gG_{a, b}(d)$: 
\begin{align*}
	G_{a, b}(d) \cneq 
	(\Hom(A, V) \oplus \Hom(V, B) \oplus \End(V))\ssslash \GL(V). 
	\end{align*}
We have the diagram 
\begin{align}\label{dia:Gflip}
	\xymatrix{
G_{a, b}^+(d) \ar[rd] \ar@{.>}[rr]& & G_{a, b}^-(d) \ar[ld] \\
& G_{a, b}(d) &	
}
	\end{align}
which is a flip if $a>b>0$, flop if $a=b>0$ (see~\cite[Lemma~7.11]{Toddbir}). 

\begin{rmk}\label{rmk:b=0}
	When $b=0$, then 
	$G_{a, 0}^-(d)=\emptyset$ and 
	$G_{a, 0}^+(d)$ is the Quot scheme on $\mathbb{A}^1$
	which parametrizes 
	quotients $\oO_{\mathbb{A}^1}^{\oplus a} \twoheadrightarrow Q$ 
	such that $Q$ is zero-dimensional of length $d$. 
	\end{rmk}
We take the Weyl-invariant norm on $N_{\mathbb{R}}=\mathbb{R}^d$
to be $\lvert \lambda \rvert^2=\lambda_1^2+\cdots+\lambda_d^2$. 
By~\cite[Lemma~5.1.9]{TocatDT}, 
	we have the KN stratifications with respect to $(\chi_0^{\pm}, \lvert \ast \rvert)$
	\begin{align}\label{KN:Gst}
		\gG_{a, b}(d)=
		\sS_0^{\pm} \sqcup \sS_1^{\pm} \sqcup \cdots \sqcup \sS_{d-1}^{\pm} 
		\sqcup G_{a, b}^{\pm}(d)
	\end{align}
	where $\sS_i^{+}$ consists of $(\alpha, \beta, \gamma)$
	such that the image of $\alpha \colon A \to V$
	generates $i$-dimensional 
	$\mathbb{C}[Q]$-submodule of $V_{\gamma}$, 
	$\sS_i^{-}$ consists of 
	$(\alpha, \beta, \gamma)$ such that the 
	image of $\beta^{\vee} \colon B^{\vee} \to V^{\vee}$
	generates $i$-dimensional 
	$\mathbb{C}[Q]$-submodule of $V^{\vee}_{\gamma^{\vee}}$. 
	The associated one parameter subgroups 
	$\lambda_i^{\pm} \colon \C \to T$ are taken as 
		\begin{align}\label{lambdai}
		\lambda_i^{+}(t)=(\overbrace{1, \ldots, 1}^i, \overbrace{t^{-1}, \ldots, t^{-1}}^{d-i}), \ 
		\lambda_i^-(t)=(\overbrace{t, \ldots, t}^{d-i}, \overbrace{1, \ldots, 1}^i)
	\end{align}
with associated slopes (\ref{slope}) to be $\mu_i^{\pm}=\sqrt{d-i}$. 
	
	\subsection{Window subcategories}\label{subsec:Gflip:2}
	For $c \in \mathbb{Z}$, we set 
	\begin{align}\label{Bcd}
		\mathbb{B}_{c}(d) \cneq \{(x_1, x_2, \ldots, x_d) \in M^+ : 
		0 \le x_i \le c-1\}. 
	\end{align}
	We define the triangulated subcategory 
	\begin{align}\label{window:Wc}
		\mathbb{W}_c(d) \subset D^b(\gG_{a, b}(d))
	\end{align}
	to be the smallest thick triangulated subcategory 
	which contains $V(\chi) \otimes \oO_{\gG_{a, b}(d)}$
	for $\chi \in \mathbb{B}_c(d)$. 
	Note that $V(\chi)$ is the Schur power of $V$
	associated with the Young diagram corresponding to $\chi$. 
	\begin{prop}\label{prop:WGD}
		The following composition functors are equivalences
		\begin{align}\label{compose:W}
		 &\mathbb{W}_{a}(d) \subset D^b(\gG_{a, b}(d)) \twoheadrightarrow 
			D^b(G_{a, b}^+(d)), \\
			\notag	&\mathbb{W}_{b}(d) \subset D^b(\gG_{a, b}(d)) \twoheadrightarrow 
			D^b(G^-_{a, b}(d)).
		\end{align}
	\end{prop}
\begin{proof}
Let $\lambda_i^+$ be the one parameter 
subgroup in (\ref{lambdai}). 
Then $\eta_i^+$ given in (\ref{etai}) is 
\begin{align*}
	\eta_i^+ &=\langle \lambda_i^+, 
	(\Hom(A, V)^{\vee} \oplus \Hom(V, B)^{\vee} \oplus \End(V)^{\vee})^{\lambda_i^+>0}
	-\End(V)^{\lambda_i^+>0} \rangle \\
	&=a(d-i). 
\end{align*}
Let $\chi'=(x_1', \ldots, x_d')$ be a $T$-weight of $V(\chi)$ for $\chi \in \mathbb{B}_a(d)$. 
Then we have $0\le x_j' \le a-1$ for $1\le j\le d$, 
so
\begin{align*}
	-\eta_i^+ =-a(d-i)<
	\langle \chi', \lambda_i^+ \rangle 
	=-\sum_{j=i+1}^d x_j' \le 0.
\end{align*}
Therefore by setting $m_i=-\eta_i^+ + \varepsilon$ for $0<\varepsilon \ll 1$
and $l=\chi_0$ in (\ref{window:m}), we have
\begin{align}\label{incl:Wab}
	\mathbb{W}_a(d) \subset \mathbb{W}_{m_{\bullet}}^{\chi_0}(\gG_{a, b}(d)) \subset 
	D^b(\gG_{a, b}(d)). 
\end{align} 
It follows that the first composition functor in (\ref{compose:W}) 
is fully-faithful. 
A similar argument also shows that, by setting $m_i=0$, 
the second composition functor in (\ref{compose:W}) is 
also fully-faithful. 

It remains to show that (\ref{compose:W}) are essentially surjective. 
By setting $a=b$ in (\ref{Yab}), 
we have the symmetric $\GL(V)$-representation
$Y_{a, a}$.
Note that $\gG_{a, a}(d)=[Y_{a, a}/\GL(V)]$. 
The subset of weights (\ref{sub:W}) for $Y_{a, a}$ 
is given by 
\begin{align*}
	\Sigma_a(d)=\sum_{-a\le c_{i}\le a}c_i \cdot
	 e_i+\sum_{-1\le c_{ij} \le 1, i>j}c_{ij} \cdot (e_i-e_j). 
	\end{align*}
Therefore we have 
\begin{align*}
	\frac{1}{2}\Sigma_a(d)-\rho+t\chi_0=\sum_{t-a/2\le c_{i}\le t+a/2}c_i \cdot
	e_i+\sum_{-1\le c_{ij} \le 0, i>j}c_{ij} \cdot (e_i-e_j). 
\end{align*}
	Therefore any element $\chi$ in the above set is written as 
\begin{align*}
	\chi=(-c_{21}-\cdots -c_{d1}+c_1)e_1+
	(c_{21}-c_{32}-\cdots -c_{d2}+c_2)e_2+\cdots 
	+(c_{d1}+\cdots +c_{d d-1}+c_d)e_d. 
\end{align*}
We write it as $\chi=\alpha_1 e_1+ \cdots + \alpha_d e_d$. 
If it lies in $M^+_{\mathbb{R}}$, then 
as $c_{ij}\le 0$ we have 
\begin{align*}
	t-\frac{a}{2} \le 
	c_1 \le \alpha_1 \le \cdots \le \alpha_d \le c_d \le t+\frac{a}{2}.
	\end{align*}
It follows that we have 
\begin{align*}
	\left(\frac{1}{2}\Sigma_a(d)-\rho+t\chi_0 \right) \cap M_{\mathbb{R}}^+
	=\left(\sum_{t-a/2\le c_{i}\le t+a/2}c_i \cdot
	e_i\right) \cap M_{\mathbb{R}}^+. 
	\end{align*}
By setting
$t=a/2-\varepsilon$ for $0<\varepsilon \ll 1$, we conclude 
that 
\begin{align}\label{Sigma:a}
	\left(\frac{1}{2}\Sigma_a(d)-\rho+t\chi_0 \right) \cap M^+
	=\{(x_1, x_2, \ldots, x_d)\in M: 0\le x_1 \le \cdots \le x_d \le a-1\}.
	\end{align}
Since $\chi_0$ is $\Sigma_a(d)$-generic
and $\Sigma_a(d)/2-\rho+t\chi_0$ does not contain 
integer points on the boundary, 
the composition (\ref{compose:W}) is an equivalence by Theorem~\ref{compose:W}
when $a=b$. 

When $a>b$, we fix a decomposition into the direct sum $A=B \oplus B'$. 
Then the projection $A \twoheadrightarrow B$ and the inclusion $B\hookrightarrow A$
define the projections $p$ and the zero sections $i$
\begin{align*}
	\xymatrix{
	\gG_{a, a}(d) \ar[r]_-{p} & \ar@/_10pt/[l]_-{i} \gG_{a, b}(d) \ar[r]_-p &
	\gG_{b, b}(d)  \ar@/_10pt/[l]_-{i}. 
}
	\end{align*}
Since the $\chi_0$-stability (resp.~$\chi_0^{-1}$-stability) on $\gG_{a, b}(d)$ 
does not impose constraint 
on $\Hom(V, B)$-factor (resp.~$\Hom(A, V)$-factor), 
we have Cartesian squares
\begin{align*}
	\xymatrix{
G_{a, a}^+(d) \ar[r]_-{p} \ar@<-0.3ex>@{^{(}->}[d]  \diasquare & G_{a, b}^+(d) 
\ar@<-0.3ex>@{^{(}->}[d]  \ar@/_10pt/[l]_-{i}\\
\gG_{a, a}(d) \ar[r]_-{p} & \gG_{a, b}(d) \ar@/_10pt/[l]_-{i}, 
}
\quad 
	\xymatrix{
	G_{a, b}^-(d) \ar[r]_-{p} \ar@<-0.3ex>@{^{(}->}[d] \diasquare & G_{b, b}^-(d) \ar@<-0.3ex>@{^{(}->}[d] \ar@/_10pt/[l]_-{i}\\
	\gG_{a, b}(d) \ar[r]_-{p} & \gG_{b, b}(d) \ar@/_10pt/[l]_-{i}.
}
	\end{align*}
Here the vertical arrows are open immersions. 
Note that each morphism $p$ is an affine bundle. 
We have the functors
\begin{align*}
	i^{\ast} \colon D^b(G_{a, a}^+(d)) \to D^b(G_{a, b}^+(d)), \ 
	p^{\ast} \colon D^b(G_{b, b}^-(d)) \to D^b(G_{a, b}^-(d)). 
	\end{align*}
Since the images of the above functors generate
$D^b(G_{a, b}^+(d))$, $D^b(G_{a, b}^-(d))$ respectively, 
from the essentially surjectivity of the functors (\ref{compose:W})
for $a=b$, we also have the essentially surjectivity of (\ref{compose:W})
for $a>b$. 
\end{proof}

\begin{rmk}\label{rmk:Wab}
	The proof of Proposition~\ref{prop:WGD}
	implies that the first inclusion in (\ref{incl:Wab}) is an equal, 
	i.e. 
	\begin{align*}
		\mathbb{W}_a(d)=\mathbb{W}_{m_i=-\eta_i^+ +\varepsilon}(\gG_{a, b}(d)), \ 
		\mathbb{W}_b(d)=\mathbb{W}_{m_i=0}(\gG_{a, b}(d)). 
		\end{align*}
	This fact will be used in Lemma~\ref{lem:tauy}. 
	\end{rmk}

\subsection{Computations of categorified Hall products}
Let $\lambda \colon \C \to T \subset \GL(V)$ be the one parameter subgroup 
given by 
\begin{align}\label{set:lambda2}
	\lambda(t)=(\overbrace{t, \ldots, t}^{d_1}, 1, \ldots, 1). 
	\end{align}
Then we have 
\begin{align*}
	\gG_{a, b}^{\lambda=0}(d)=
\mM_Q(d_1) \times \gG_{a, b}(d-d_1).
	\end{align*}
We have the diagram of attracting loci (\ref{dia:quiver2})
and the associated categorified Hall product 
\begin{align}\label{caH1}
	\ast \colon 
	D^b(\mM_Q(d_1)) \boxtimes D^b(\gG_{a, b}(d-d_1)) \to 
	D^b(\gG_{a, b}(d)). 
	\end{align}
By the iteration, we have the categorified Hall product 
\begin{align}\label{caH2}
	\ast \colon 
	D^b(\mM_Q(d_1)) \boxtimes \cdots \boxtimes D^b(\mM_Q(d_l))
	\boxtimes D^b\left(\gG_{a, b}\left(d-d_1-\cdots -d_l\right)\right)
	\to D^b(\gG_{a, b}(d)). 
	\end{align}
We remark that, for $A_i \in D^b(\mM_Q(d_i))$
and $B \in D^b(\gG_{a, b}(d-d_1-\cdots -d_l))$, 
 the above functor satisfies that 
\begin{align}\label{ast:chi0}
	(A_1 \ast \cdots \ast A_l \ast B) \otimes \chi_0 \cong
	(A_1 \otimes \chi_0) \ast \cdots (A_l \otimes \chi_0) \ast
	(B \otimes \chi_0). 
	\end{align}
Here $\chi_0$ in the LHS is the determinant character for $\GL(d)$
and by abuse of notation $\chi_0$ in the RHS 
are determinant characters for $\GL(d_i)$ and $\GL(d-d_1-\cdots -d_l)$. 
The above isomorphism follows immediately from the definition of
categorical Hall products. 

We fix $c>b$. 
For $d'<d$, 
we fix the following embedding 
\begin{align}\label{emb:Bc}
	\mathbb{B}_c(d') \hookrightarrow \mathbb{B}_c(d), \ 
	(x_{d-d'+1}, \ldots, x_d) \mapsto (x_1=\cdots=x_{d-d'}=0, x_{d-d'+1}, \ldots, x_d). 
\end{align}
We regard an element of $\mathbb{B}_c(d')$ as an element of $\mathbb{B}_c(d)$
by the above embedding. 
For $0\le k\le d$, 
let $\mathbb{B}_{c, k}(d) \subset \mathbb{B}_{c}(d)$ be the subset defined by 
\begin{align*}
	\mathbb{B}_{c, k}(d)=\{(x_1, \ldots, x_d) \in \mathbb{B}_c(d) : 
	x_1=\cdots=x_k=0, x_{k+1}>0\}. 
	\end{align*}
Note that we have $\mathbb{B}_{c, d}(d)=\{0\}$
and 
\begin{align*}
	\mathbb{B}_{c, 0}(d)=\mathbb{B}_{c-1}(d)+\chi_0
	\end{align*}
as $\chi_0=(1, 1, \ldots, 1)$. 
We have the decomposition into the disjoint union 
\begin{align*}
	\mathbb{B}_c(d)=\mathbb{B}_{c, 0}(d) \sqcup \mathbb{B}_{c, 1}(d) \sqcup \cdots 
	\sqcup \mathbb{B}_{c, d}(d). 
	\end{align*}
For $d'<d$ and $d-d'\le k\le d$, the embedding (\ref{emb:Bc}) induces the bijection 
\begin{align}\label{bij:k}
	\mathbb{B}_{c, k-d+d'}(d') \stackrel{\cong}{\to}\mathbb{B}_{c, k}(d).
	\end{align}
We define the subcategory 
\begin{align*}
	\mathbb{W}_{c, k}(d) \subset \mathbb{W}_{c}(d)
	\end{align*}
to be split generated by $V(\chi) \otimes \oO_{\gG_{a, b}(d)}$
for $\chi \in \mathbb{B}_{c, k}(d)$. 

\begin{prop}\label{prop:Hall}
	For $1\le k\le d$ and $\chi \in \mathbb{B}_{c, 0}(d-k)$, 
	the object 
	\begin{align*}
		\oO_{\mM_Q(k)} \ast (V(\chi) \otimes \oO_{\gG_{a, b}(d-k)})
		\in D^b(\gG_{a, b}(d))
		\end{align*}
	is generated by 
	$V(\chi) \otimes \oO_{\gG_{a, b}(d)}$, where $\chi$ is 
	regarded as an element of $\mathbb{B}_{c, k}(d)$ by (\ref{bij:k}), 
	 and $V(\chi') \otimes \oO_{\gG_{a, b}(d)}$
	for $\chi' \in \mathbb{B}_{c, <k}(d)$. 
	Moreover $V(\chi) \otimes \oO_{\gG_{a, b}(d)}$ appears exactly once. 
	\end{prop}
	\begin{proof}
		Let $Y_{a, b}(d)$ be the $\GL(V)$-representation (\ref{Yab}), 
		and $\lambda$ the one parameter subgroup (\ref{set:lambda2}) 
		for $d_1=k$. 
		Then we have 
		\begin{align*}
			\{ \beta \in \wt_T(Y_{a, b}(d)) : 
			\langle \lambda, \beta <0 \rangle \}
			=\bigcup_{1\le i\le k} \{\overbrace{-e_i, \ldots, -e_i}^b\}
			\bigcup_{1\le j\le k, k<i\le d}
			\{(e_i-e_j)\}. 
			\end{align*}
		Let $I$ be a subset of weights in the above set. 
		Then in the notation of Proposition~\ref{prop:resol}, 
		for $\chi=(0, \ldots, 0, x_{k+1}, \ldots, x_d)
		\in \mathbb{B}_{c, k}(d)$, the element $\chi-\sigma_I+\rho$ 
		is of the form 
		\begin{align}\label{chi:rho}
			\chi-\sigma_I+\rho=
		\sum_{i=k+1}^d x_i e_i +\sum_{i=1}^k s_i e_i 
			-\sum_{k<i\le d, 1\le j\le k}s_{ij}(e_i-e_j) +\frac{1}{2}\sum_{i>j}(e_i-e_j)
			\end{align}
		for some 
		$s_{i} \in \mathbb{Z}$ with $0\le s_i \le b$
		and $s_{ij} \in \{0, 1\}$. 
	Therefore we have 
	\begin{align}\label{chisigma}
		\chi-\sigma_I+\rho \in 
		\left(\sum_{-c/2 \le c_i \le c/2}c_i e_i+
		\sum_{i>j, -1/2 \le c_{ij} \le 1/2}
		c_{ij}(e_i-e_j) \right) + t\chi_0
		\end{align}	
		for $t=c/2-\varepsilon$ with $0<\varepsilon \ll 1$. 
		Suppose that $(\chi-\sigma_I)^+ \in M^+$ is defined. 
		By its definition, 
		there is unique $w \in S_d$
		such that  
		\begin{align}\label{id:Weyl}
			w(\chi-\sigma_I+\rho)
			=(\chi-\sigma_I)^++\rho. 
			\end{align}
		Since the right hand side of (\ref{chisigma})
		is invariant under the Weyl group action, 
		we have 
		\begin{align*}
			(\chi-\sigma_I)^+
			&\in \left(\sum_{t-c/2 \le c_i \le t+c/2}c_i e_i+
			\sum_{i>j, -1 \le c_{ij} \le 0}
			c_{ij}(e_i-e_j) \right) \cap M^+ \\
			&=\{(x_1'', \ldots, x_d'') \in M : 
			0\le x_1''\le \cdots \le x_d'' \le c-1\}. 
			\end{align*}
		Here the last identity follows as in (\ref{Sigma:a}).
		Therefore we have $(\chi-\sigma_I)^+ \in \mathbb{B}_c(d)$. 
		
		We show that $(\chi-\sigma_I)^+ \in \mathbb{B}_{c, \le k}(d)$, 
		and $(\chi-\sigma_I)^+ \in \mathbb{B}_{c, k}(d)$ if and 
		only if $I=\emptyset$. 
		Then the proposition follows from Proposition~\ref{prop:resol}.
		  Let us write 
		  \begin{align*}
		  	(\chi-\sigma_I)^+=\sum_{i=1}^d x_i' e_i, \ 
		  	0\le x_1' \le \cdots \le x_{d}' \le c-1. 
		  	\end{align*}
	  	Then by setting $y_i=x_i'-d/2-1/2+i$, we have 
	  	\begin{align*}
	  		(\chi-\sigma_I)^++\rho=\sum_{i=1}^d y_i e_i, \ 
	  	-\frac{d-1}{2} \le y_1<y_2<\cdots <y_d <c+\frac{d-1}{2}. 
	  	\end{align*}
  	Note that $(\chi-\sigma_I)^+ \in \mathbb{B}_{c, 0}(d)$ if and only 
  	if $y_i>-(d-1)/2$ for all $i$. 
  	
  	Let $(\chi-\sigma_I+\rho)_{e_i}$ be the coefficient of 
  	$(\chi-\sigma_I+\rho)$ at $e_i$. 
  	From (\ref{chi:rho}), for $1\le j\le k$ we have 
  	\begin{align}\label{ineq:chi1}
  		(\chi-\sigma_I+\rho)_{e_j}
  		=-\frac{d}{2}-\frac{1}{2}+j+s_j+\sum_{i=k+1}^d s_{ij} 
  		\ge -\frac{d-1}{2} 
  		\end{align}
  	and the equality holds if and only if $j=1$, $s_1=0$ and $s_{i1}=0$ for 
  	all $k<i \le d$. 
  	Also for $k<i\le d$, we have 
  	\begin{align}\label{ineq:chi2}
  			(\chi-\sigma_I+\rho)_{e_i}
  		=-\frac{d}{2}-\frac{1}{2}+i +x_i-\sum_{j=1}^k s_{ij}>-\frac{d-1}{2}
  		+i-k-1 \ge -\frac{d-1}{2}. 
  		\end{align}
  	Here the first inequality is strict since $x_i>0$. 
  	Therefore by the identity (\ref{id:Weyl}), 
  	we have either $(\chi-\sigma_I)^+ \in \mathbb{B}_{c, 0}(d)$, 
  	or $s_1=s_{i1}=0$ for all $k<i\le d$. 
  	
  	Suppose that 
  	$s_1=s_{i1}=0$ for all $k<i\le d$, so that 
  	$y_1=-(d-1)/2$. 
  	For $2\le j\le k$, 
  	from (\ref{ineq:chi1}) we have 
  	\begin{align*}
  		(\chi-\sigma_I+\rho)_{e_j} \ge -\frac{d-3}{2}
  		\end{align*}
  	and the equality holds only if $j=2$, $s_2=0$ 
  	and $s_{i2}=0$ for $k<i\le d$. 
  	Moreover for $k<i\le d$, the inequality (\ref{ineq:chi2}) is improved 
  	as 
  	\begin{align*}
  		(\chi-\sigma_I+\rho)_{e_i}
  	=-\frac{d}{2}-\frac{1}{2}+i +x_i-\sum_{j=2}^k s_{ij}>-\frac{d-3}{2}
  	+i-k-1 \ge -\frac{d-3}{2}.
  	\end{align*}
  	It follows that we have either
  	$y_2>-(d-3)/2$, i.e. 
  	 $(\chi-\sigma_I)^+ \in \mathbb{B}_{c, 1}(d)$,
  	or $s_2=s_{i2}=0$ for all $k<i\le d$. 
  	
  	Repeating the above argument, we conclude that 
  	$(\chi-\sigma_I)^+ \in \mathbb{B}_{c, <k}(d)$, 
  	or $s_1=\cdots =s_k=0$ and $s_{ij}=0$ for all $1\le j\le k$
  	and $k<i\le d$. 
  	In the latter case, we have $I=\emptyset$ and 
  	$(\chi-\sigma_I)^+=\chi \in \mathbb{B}_{c, k}(d)$. 
  	\end{proof}
  	
  	\begin{lem}\label{lem:gen}
  		The subcategory 
  		$\mathbb{W}_c(d) \subset D^b(\gG_{a, b}(d))$ is generated by 
  		$\mathbb{W}(k) \ast (\mathbb{W}_{c-1}(d-k) \otimes \chi_0)$
  		for $0\le k\le d$. 
  		\end{lem}
  	\begin{proof}
  		Since $\mathbb{W}_{c-1}(d-k) \otimes \chi_0=\mathbb{W}_{c, 0}(d-k)$
  		and $\mathbb{W}(k)$ is generated by $\oO_{\mM_Q(k)}$ by 
  		Lemma~\ref{lem:genO}, we have 
  		\begin{align}\label{incl:W}
  			\mathbb{W}(k) \ast (\mathbb{W}_{c-1}(d-k) \otimes \chi_0)
  			\subset \mathbb{W}_c(d), \ 0\le k\le d
  			\end{align}
  		by Proposition~\ref{prop:Hall}. 
  		It is enough to show that for 
  		any $\chi \in \mathbb{B}_c(d)$
  		the object $V(\chi) \otimes \oO_{\gG_{a, b}}(d)$
  		 is generated by the 
  		LHS in (\ref{incl:W}) for $0\le k\le d$. 
  		If $\chi \in \mathbb{B}_{c, 0}(d)$, 
  		then $V(\chi) \otimes \oO_{\gG_{a, b}}(d)$
  		is an object in the LHS in (\ref{incl:W}) for $k=0$. 
  		For $\chi \in \mathbb{B}_{c, k}(d)$ with $k>0$, 
  		by Proposition~\ref{prop:Hall}
  		$V(\chi) \otimes \oO_{\gG_{a, b}}(d)$ is generated by 
  		$\mathbb{W}(k) \ast \mathbb{W}_{c, 0}(d-k)=\mathbb{W}(k) \ast (\mathbb{W}_{c-1}(d-k) \otimes \chi_0)$
  		and $V(\chi') \otimes \oO_{\gG_{a, b}(d)}$ for $\chi' \in 
  		\mathbb{B}_{c, <k}(d)$. 
  		Therefore by the induction of $k$, 
  		$V(\chi) \otimes \oO_{\gG_{a, b}}(d)$ is generated by the LHS in (\ref{incl:W}). 
  				\end{proof}
	
	\begin{prop}\label{prop:gen}
		The subcategory 
		$\mathbb{W}_c(d) \subset D^b(\gG_{a, b}(d))$
		is generated by the subcategories 
		\begin{align}\label{Wdi}
			\mathbb{W}(d_{\bullet}) \cneq 
			\mathbb{W}(d_1) \ast (\mathbb{W}(d_2) \otimes \chi_0) \ast
			\cdots \ast (\mathbb{W}(d_l) \otimes \chi_0^{l-1}) \ast 
			(\mathbb{W}_b(d-d_1-\cdots -d_l) \otimes \chi_0^{l})
			\end{align}
		for $l=c-b >0$ and 
		 $(d_1, \ldots, d_l) \in \mathbb{Z}_{\ge 0}^l$. 
				\end{prop}
	\begin{proof}
		Suppose that the proposition holds for $c-1$.
		Then for any $d_1 \ge 0$, the category 
		$\mathbb{W}_{c-1}(d-d_1)$ is generated by 
		\begin{align*}
			\mathbb{W}(d_2) \ast (\mathbb{W}(d_3) \otimes \chi_0) \ast
			\cdots \ast (\mathbb{W}(d_l) \otimes \chi_0^{l-2}) \ast 
			(\mathbb{W}_b(d-d_1-\cdots -d_l) \otimes \chi_0^{l-1})
		\end{align*}
	for $(d_2, \ldots, d_l) \in \mathbb{Z}_{\ge 0}^{l-1}$. 
	Then by Lemma~\ref{lem:gen}, 
	$\mathbb{W}_c(d)$ is generated by 
	\begin{align*}
&\mathbb{W}(d_1) \ast	\left\{	\left(\mathbb{W}(d_2) \ast (\mathbb{W}(d_3) \otimes \chi_0) \ast
		\cdots \ast (\mathbb{W}(d_l) \otimes \chi_0^{l-2}) \ast 
		(\mathbb{W}_b(d-d_1-\cdots -d_l) \otimes \chi_0^{l-1}) \right) \otimes \chi_0
		\right\} \\
		&=	\mathbb{W}(d_1) \ast (\mathbb{W}(d_2) \otimes \chi_0) \ast
		\cdots \ast (\mathbb{W}(d_l) \otimes \chi_0^{l-1}) \ast 
		(\mathbb{W}_b(d-d_1-\cdots -d_l) \otimes \chi_0^{l})
		\end{align*}
	for
	$(d_1, d_2, \ldots, d_l) \in \mathbb{Z}_{\ge 0}^l$. 
	Then the proposition holds by the induction of $c$. 
		\end{proof}
	
	\subsection{Semiorthogonal decomposition}
	\begin{prop}\label{prop:sod}
		For each $0\le k\le d$, the functor 
		\begin{align}\label{ast:ff}
		\ast \colon \mathbb{W}(k) \boxtimes 
		(\mathbb{W}_{c-1}(d-k) \otimes \chi_0)) \to 
		\mathbb{W}_c(d)
		\end{align}
	is fully-faithful, such that we have
	the semiorthogonal decomposition 
	\begin{align*}
		\mathbb{W}_c(d)=
		\langle \mathbb{W}(d) \ast (\mathbb{W}_{c-1}(0) \otimes \chi_0), \ 
		\mathbb{W}(d-1) \ast (\mathbb{W}_{c-1}(1) \otimes \chi_0), \cdots, 
		\mathbb{W}_{c-1}(d) \otimes \chi_0 
		  \rangle. 
		\end{align*}		
		\end{prop}
	\begin{proof}
		The generation is proved in Lemma~\ref{lem:gen}. 
		It is enough to show that the functor (\ref{ast:ff}) is fully-faithful, 
		and the images of the above functors are semiorthogonal. 
		
		Let us take $k\le k'$ and 
		$\chi \in \mathbb{B}_{c-1, 0}(d-k)$, 
		$\chi' \in \mathbb{B}_{c-1, 0}(d-k')$. 
		Let $\lambda$, $\lambda'$ be the 
		one parameter subgroups $\C \to T \subset \GL(V)$
		given by 
		\begin{align*}
			\lambda(t)=(\overbrace{t, \ldots, t}^{k}, 1, \ldots, 1), \ 
				\lambda'(t)=(\overbrace{t, \ldots, t}^{k'}, 1, \ldots, 1). 
			\end{align*}
		We have the diagrams of attracting loci 
		\begin{align}\label{dia:quiver3}
			\xymatrix{
				\gG_{a, b}(d)^{\lambda \ge 0} \ar[r]^-{p_{\lambda}} \ar[d]_-{q_{\lambda}} & \gG_{a, b}(d) \\
			\mM_Q(k) \times \gG_{a, b}(d-k), & 	
			} \quad 
		\xymatrix{
			\gG_{a, b}(d)^{\lambda \ge 0} \ar[r]^-{p_{\lambda'}} \ar[d]_-{q_{\lambda'}} & \gG_{a, b}(d) \\
			\mM_Q(k') \times \gG_{a, b}(d-k'). & 	
		}		
		\end{align}
		Then we have 
		\begin{align}\notag
			&\Hom(\oO \ast (V(\chi) \otimes \oO), \oO \ast (V(\chi') \otimes \oO)) \\
	\notag	&\cong \Hom(p_{\lambda\ast}q_{\lambda}^{\ast}(\oO \boxtimes 
			(V(\chi) \otimes \oO), p_{\lambda'\ast}q_{\lambda'}^{\ast}(\oO \boxtimes 
			(V(\chi') \otimes \oO)) \\
		\label{isom:V}	&\cong \Hom(p_{\lambda'}^{\ast}p_{\lambda\ast}q_{\lambda}^{\ast}(\oO \boxtimes 
			(V(\chi) \otimes \oO), q_{\lambda'}^{\ast}(\oO \boxtimes 
			(V(\chi') \otimes \oO)).
			\end{align}
		The object $p_{\lambda\ast}q_{\lambda}^{\ast}(\oO \boxtimes 
		(V(\chi) \otimes \oO)$
		is generated by $V(\chi) \otimes \oO$
		and $V(\chi'') \otimes \oO$ for $\chi'' \in \mathbb{B}_{c, <k}(d)$
		by Proposition~\ref{prop:Hall}. 
		If $k<k'$, then 
		any $T$-weight of $V(\chi)$ and $V(\chi'')$ pairs positively with 
		$\lambda'$. 
		Since any $T$-weight of $V(\chi')$ pairs zero with $\lambda'$, 
		it follows that (\ref{isom:V}) is zero by Lemma~\ref{lem:vanish} (i).
		Therefore the images of the functors (\ref{prop:sod}) are semiorthogonal. 
		
		Suppose that $k=k'$. 
		Since any $T$-weight of $V(\chi'')$ pairs positively with $\lambda$, 
		we have 
		\begin{align*}
			(\ref{isom:V})&\cong 
		\Hom(p_{\lambda'}^{\ast}(V(\chi) \otimes \oO), q_{\lambda}^{\ast}(\oO \boxtimes 
		(V(\chi') \otimes \oO))) \\
		&\cong \Hom(q_{\lambda'}^{\ast}(\oO \boxtimes 
		(V(\chi) \otimes \oO)), q_{\lambda}^{\ast}(\oO \boxtimes 
		(V(\chi') \otimes \oO))) \\
		&\cong \Hom(\oO \boxtimes 
		(V(\chi) \otimes \oO), \oO \boxtimes 
		(V(\chi') \otimes \oO)). 		
			\end{align*}
	Here the 
	second isomorphism follows since
	$\gG_{a, b}(d)^{\lambda \ge 0}$ is the stack of exact sequences 
	\begin{align*}
		0 \to \mathbb{V}^{\lambda >0} \to \mathbb{V} \to \mathbb{V}^{\lambda=0} \to 0
		\end{align*}
	of $Q_{a, b}$-representations,  
	so the only Schur power power from $\mathbb{V}^{\lambda=0}$ survives by 
	Lemma~\ref{lem:vanish} (i).
	The 	
	last isomorphism follows from Lemma~\ref{lem:vanish} (ii). 
	Therefore the functor (\ref{ast:ff}) is fully-faithful. 
		\end{proof}
	
	For $d_{\bullet}=(d_1, \ldots, d_l) \in \mathbb{Z}_{\ge 0}^l$
	and $d_{\bullet}'=(d_1', \ldots, d_l') \in \mathbb{Z}_{\ge 0}^l$, 
	we take the following lexicographic order: 
	$d_{\bullet}' \succ d_{\bullet}$ if and only if 
	there is some $m$ so that $d_i=d_i'$ for $i<m$ and $d_m>d_m'$. 
	\begin{thm}\label{thm:SOD}
		We take $c>b$ with $l \cneq c-b>0$. Then for 
		each $d_{\bullet}=(d_1, \ldots, d_l) \in \mathbb{Z}_{\ge 0}^l$, 
		the categorified Hall product (\ref{caH2}) 
		induces the fully-faithful functor 
		\begin{align*}
			\ast \colon 
			\mathbb{W}(d_1) \boxtimes (\mathbb{W}(d_2) \otimes \chi_0)
			\boxtimes \cdots \boxtimes (\mathbb{W}(d_l) \otimes \chi_0^{l-1})
			\boxtimes (\mathbb{W}_b(d-\sum_{i=1}^l d_i) \otimes \chi_0^l)
			\to \mathbb{W}_c(d)
			\end{align*}
		such that, 
		by setting 
		$\cC(d_{\bullet}) \subset D^b(\gG_{a, b}(d))$ to be the 
		essential image of the above functor, 
		we have the semiorthogonal decomposition 
		\begin{align*}
			\mathbb{W}_c(d)=
			\langle \cC(d_{\bullet}) : 
			d_{\bullet} \in \mathbb{Z}_{\ge 0}^l \rangle. 
			\end{align*}
		Here $\Hom(\cC(d'_{\bullet}), \cC(d_{\bullet}))=0$
		for $d_{\bullet}' \succ d_{\bullet}$. 
		\end{thm}
	\begin{proof}
		The result is proved for $c=b+1$ in Proposition~\ref{prop:sod}. 
		Then similarly to the proof of Proposition~\ref{prop:gen}, the theorem 
		easily follows by the induction of $c$. 
		\end{proof}
	
	By applying Theorem~\ref{thm:SOD} to $c=a$ and $r=a-b$, 
	we obtain the following: 
	\begin{cor}\label{cor:sod}
		There is a semiorthogonal 
		decomposition of the form 
		\begin{align*}
			D^b(G_{a, b}^+(d))
			=\left\langle D^b(M_Q(d_1)) \boxtimes \cdots 
			\boxtimes D^b(M_Q(d_r)) \boxtimes D^b(G_{a, b}^-(d-\sum_{i=1}^r d_i)) : (d_1, \ldots, d_r) \in \mathbb{Z}_{\ge 0}^r \right\rangle. 
			\end{align*}
		\end{cor}
	\begin{proof}
		The corollary follows from Theorem~\ref{thm:SOD}, 
		Proposition~\ref{prop:WGD}
		and Lemma~\ref{lem:pull}. 
		\end{proof}
	
	\subsection{The case of multiple quivers}
	For $m \ge 1$, let $Q^{(m)}$ be the quiver with 
	$m$-vertices $\{1, \ldots, m\}$
	with one loop at each vertex. 
	We also define $Q_{a, b}^{(m)}$ to be the quiver 
	with vertices $\{\infty, 1, \ldots, m\}$
	with one loop at each vertex $\{1, \ldots, m\}$, 
	$a$-arrows from $\infty$ to each $i \in \{1, \ldots, m\}$, 
	$b$-arrows from each $i \in \{1, \ldots, m\}$ to $\infty$. 
	See the following picture for $Q_{2, 1}^{(3)}$: 
	\[
	Q_{2, 1}^{(3)}=
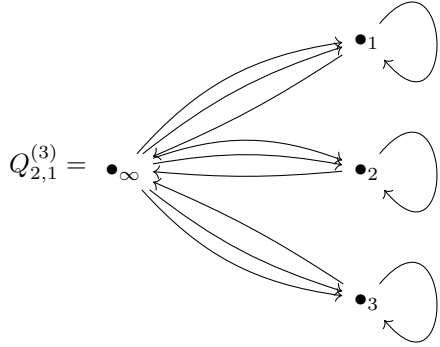
	\begin{tikzcd}
	&	& & \bullet_1 	\arrow[out=50, in=310, loop]	\arrow[llldd, bend left=5] \\ 
	& & & \\		
		\bullet_{\infty}
		\arrow[rrruu,bend left=20]
		\arrow[rrruu,bend left=10]
			\arrow[rrr,bend left=20]
		\arrow[rrr,bend left=10]
				\arrow[rrrdd,bend right=20]
					\arrow[rrrdd,bend right=10]
		& & &
		\bullet_2 \arrow[lll, bend left=5]
		\arrow[out=50, in=310, loop] \\
		& & & \\
	&	& &
		\bullet_3 \arrow[llluu, bend right=5]
			\arrow[out=50, in=310, loop]
		\end{tikzcd}
	\]
	
	Let $d^{(\ast)}=(d^{(1)}, \ldots, d^{(m)}) \in \mathbb{Z}_{\ge 0}^{m}$
	be a dimension vector of a $Q^{(m)}$-representation. 
	We set
	\begin{align*}
		\lvert d^{(\ast)} \rvert \cneq 
		d^{(1)}+\cdots+ d^{(m)}. 
		\end{align*}
	The moduli stack of $Q^{(m)}$-representations with 
	dimension vector $d^{(\ast)}$ is given by 
	\begin{align*}
		\mM_{Q^{(m)}}(d^{(\ast)})\cneq \prod_{j=1}^m 
		\mM_Q(d^{(j)})=\prod_{j=1}^{m}[\End(V^{(j)})/\GL(V^{(j)})] 
		\end{align*}
	where $\dim V^{(j)}=d^{(j)}$. 
	The $\C$-rigidified moduli stack of 
	$Q_{a, b}^{(m)}$-representations with dimension 
	vector $(1, d^{(\ast)})$ is given by 
	\begin{align*}
		\gG_{a, b}^{(m)}(d^{(\ast)}) \cneq \prod_{j=1}^m 
		\gG_{a, b}(d^{(j)}). 
		\end{align*}
	Let $\chi_0$ be the determinant character for 
	$\prod_{j=1}^m \GL(V^{(j)})$, i.e. 
	\begin{align*}
		\chi_0 \colon \prod_{j=1}^m \GL(V^{(j)}) \to \C, \ 
		\{g^{(j)}\}_{1\le j\le m} \mapsto \prod_{j=1}^m \det(g^{(j)}). 
		\end{align*}
	Then the $\chi_0^{\pm}$-stable loci are
	\begin{align*}
		G_{a, b}^{(m)\pm}(d^{(\ast)}) \cneq
		\prod_{j=1}^m G_{a, b}^{\pm}(d^{(j)}) \subset \gG_{a, b}^{(m)}(d^{(\ast)}).  
		\end{align*}
	We take the Weyl-invariant norm $\lvert \ast \rvert$
	for the cocharacter lattice of 
	$\prod_{j=1}^m\GL(V^{(j)})$ to be
	\begin{align*}
		\lvert (\lambda^{(1)}, \ldots, \lambda^{(m)}) \rvert^2
		=\lvert \lambda^{(1)} \rvert^2+ \cdots + \lvert \lambda^{(m)} \rvert^2. 
		\end{align*}
	We have the KN stratification of $\gG_{a, b}^{(m)}(d^{(\ast)})$
	with respect to $(\chi_0^{\pm 1}, \lvert \ast \rvert)$
	\begin{align}\label{KN:Gm}
		\gG_{a, b}^{(m)}(d^{(\ast)})=
		\mathscr{S}^{\pm}_0 \sqcup \cdots \sqcup \mathscr{S}^{\pm}_{\lvert d^{(\ast)} \rvert -1}
		\sqcup G_{a, b}^{(m)\pm}(d^{(\ast)}). 
		\end{align}
	A strata $\mathscr{S}_i^{\pm}$ corresponds to 
	$Q_{a, b}^{(m)}$-representations 
	such that the images of 
	arrows from the vertex $\infty$
	(resp.~duals of arrows going to $\infty$)
	generates $i$-dimensional $Q^{(m)}$-representations (see~\cite[Lemma~5.1.9]{TocatDT}). 
	Explicitly, by setting 
	\begin{align*}
		\gG_{a, b}(d^{(j)})=\sS_0^{(j)\pm} \sqcup \cdots \sqcup \sS^{(j)\pm}_{d^{(j)}-1} \sqcup 
		G_{a, b}^{\pm}(d^{(j)})
		\end{align*}
	to be KN stratifications (\ref{KN:Gst}) for 
	$d^{(j)}$, 
	we have 
	\begin{align*}
		\mathscr{S}_i^{\pm}=\coprod_{i_1+\cdots+i_m=i} \prod_{j=1}^m \sS_{i_j}^{(j)\pm}. 
		\end{align*}
	The corresponding 
	one parameter subgroup is given by 
	$\{\lambda_{i_j}^{\pm}\}_{1\le j\le m}$, 
	where $\lambda_{i_j}^{\pm}$ is given in (\ref{lambdai}), 
	with slope (\ref{slope}) given by 
	$\mu_i^{\pm}=\sqrt{\lvert d^{(\ast)}\rvert-i}$. 
	
		We set
	\begin{align}\label{box:product}
		&\mathbb{W}(d^{(\ast)}) \cneq 
		\boxtimes_{j=1}^m \mathbb{W}(d^{(j)}) \subset 
		D^b(\mM_{Q^{(m)}}(d^{(\ast)})), \\ 
	\notag	&\mathbb{W}_c(d^{(\ast)}) \cneq 
		\boxtimes_{j=1}^m \mathbb{W}_c(d^{(j)}) \subset 
		D^b(\gG_{a, b}^{(m)}(d^{(\ast)})). 
		\end{align}
	By Proposition~\ref{prop:WGD}, the following composition functors 
	are equivalences:
	\begin{align*}
		&\mathbb{W}_a(d^{(\ast)})
		\subset D^b(\gG_{a, b}^{(m)}(d^{(\ast)})) \twoheadrightarrow 
		D^b(G_{a, b}^{(m)+}(d^{(\ast)})), \\
		&\mathbb{W}_b(d^{(\ast)})
		\subset D^b(\gG_{a, b}^{(m)}(d^{(\ast)})) \twoheadrightarrow 
		D^b(G_{a, b}^{(m)-}(d^{(\ast)})). 
		\end{align*}
	
	For a decomposition 
	$d^{(\ast)}=d_1^{(\ast)}+d_2^{(\ast)}$, 
	let $\lambda \colon \C \to \prod_{j=1}^m \GL(V^{(j)})$ be the 
	one parameter subgroup given by 
	\begin{align*}
		\lambda=(\lambda^{(1)}, \ldots, \lambda^{(m)}), \ 
		\lambda^{(j)}(t)=(\overbrace{t, \ldots, t}^{d_1^{(j)}}, \overbrace{1, \ldots, 1}^{d_2^{(j)}}).
		\end{align*}
	We have the diagram of attracting loci 
	\begin{align}\label{dia:Glm}
		\xymatrix{\prod_{j=1}^m \gG_{a, b}(d^{(j)})^{\lambda^{(j)} \ge 0} 
			\ar@{=}[r] &
\gG_{a, b}^{(m)}(d^{(\ast)})^{\lambda \ge 0}
\ar[r]^-{q_{\lambda}} \ar[d]_-{p_{\lambda}} & \gG_{a, b}^{(m)}(d^{(\ast)}) \\
\mM_{Q^{(m)}}(d_1^{(\ast)})
\times 	\gG_{a, b}^{(m)}(d_2^{(\ast)}) \ar@{=}[r] & 
\gG_{a, b}^{(m)}(d^{(\ast)})^{\lambda = 0}.  &	
}
		\end{align}
	which gives the categorified 
	Hall product 
	\begin{align}\label{caHall:mult}
		\ast=q_{\lambda\ast}p_{\lambda}^{\ast} \colon 
		D^b(\mM_{Q^{(m)}}(d_1^{(\ast)})) \boxtimes 
		D^b(\gG_{a, b}^{(m)}(d_2^{(\ast)})) \to 
			D^b(\gG_{a, b}^{(m)}(d^{(\ast)})). 
		\end{align}
	From Theorem~\ref{thm:SOD}, we have the following: 
	\begin{thm}\label{thm:SOD:mult}
		We take $c>b$ with $l \cneq c-b>0$. Then for 
	each $d_{\bullet}=(d_1, \ldots, d_l) \in \mathbb{Z}_{\ge 0}^l$, 
	the categorified Hall product induces 
	the fully-faithful functor 
	\begin{align*}
		\ast \colon 
		\bigoplus_{\begin{subarray}{c}
				(d_1^{(\ast)}, \ldots, d_r^{(\ast)}) \\
				\lvert d_i^{(\ast)} \rvert=d_i
			\end{subarray}}
		\mathbb{W}(d_1^{(\ast)}) \boxtimes (\mathbb{W}(d_2^{(\ast)}) \otimes \chi_0)
		\boxtimes \cdots \boxtimes (\mathbb{W}(d_l^{(\ast)}) \otimes \chi_0^{l-1})
		&\boxtimes (\mathbb{W}_b(d^{(\ast)}-\sum_{i=1}^l d^{(\ast)}_i) \otimes \chi_0^l)\\
		&\to \mathbb{W}_c(d^{(\ast)})
	\end{align*}
	such that, 
	by setting 
	$\cC(d_{\bullet}) \subset D^b(\gG_{a, b}^{(m)}(d^{(\ast)}))$ to be the 
	essential image of the above functor, 
	we have the semiorthogonal decomposition 
	\begin{align*}
		\mathbb{W}_c(d^{(\ast)})=
		\langle \cC(d_{\bullet}) : 
		d_{\bullet} \in \mathbb{Z}_{\ge 0}^l \rangle. 
	\end{align*}
	Here $\Hom(\cC(d'_{\bullet}), \cC(d_{\bullet}))=0$
	for $d_{\bullet}' \succ d_{\bullet}$. 
\end{thm}
\begin{proof}
	The categorified Hall product (\ref{caHall:mult})
	fits into the commutative diagram 
	\begin{align*}
		\xymatrix{
		D^b(\mM_{Q^{(m)}}(d_1^{(\ast)})) \boxtimes 
	D^b(\gG_{a, b}^{(m)}(d_2^{(\ast)})) \ar[r]^-{\ast} \ar[d]_-{\sim} 
	& 	D^b(\gG_{a, b}^{(m)}(d^{(\ast)})) \ar[d]^-{\sim} \\
	\boxtimes_{j=1}^m
	\left(D^b(\mM_{Q}(d_1^{(j)})) \boxtimes 
	D^b(\gG_{a, b}(d_2^{(j)}))	\right) \ar[r]^-{\boxtimes \ast} & 
	\boxtimes_{j=1}^m 	D^b(\gG_{a, b}(d^{(j)})).
	}
		\end{align*}
	Here the left vertical arrow is the exchange of factors. 
	Therefore from Theorem~\ref{thm:SOD}, 
	for each $(d_1^{(\ast)}, \ldots, d_l^{(\ast)})$ 
	the functor 
		\begin{align*}
		\ast \colon 
		\mathbb{W}(d_1^{(\ast)}) \boxtimes (\mathbb{W}(d_2^{(\ast)}) \otimes \chi_0)
		\boxtimes \cdots \boxtimes (\mathbb{W}(d_l^{(\ast)}) \otimes \chi_0^{l-1})
		\boxtimes (\mathbb{W}_b(d^{(\ast)}-\sum_{i=1}^l d^{(\ast)}_i) \otimes \chi_0^l)
		\to D^b(\gG_{a, b}^{(m)}(d^{(\ast)}))
	\end{align*}
is fully-faithful such that, by setting 
$\cC(d_{\bullet}^{(\ast)})$ to be essential image of the above functor, we have 
the semiorthogonal decomposition of the form 
\begin{align*}
	\mathbb{W}_c(d^{(\ast)})=\langle \cC(d_{\bullet}^{(\ast)}) : 
	d_{\bullet}^{(\ast)} \in \mathbb{Z}_{\ge 0}^{lm} \rangle. 
	\end{align*}
For $d_{\bullet}^{(\ast)}\in \mathbb{Z}_{\ge 0}^{lm}$, we set 
$\lvert d_{\bullet}^{(\ast)}\rvert \cneq (\lvert d_1^{(\ast)} \rvert, \ldots, 
\lvert d_l^{(\ast)} \rvert) \in \mathbb{Z}_{\ge 0}^l$. 
Then for $d_{\bullet}^{(\ast)} \neq d_{\bullet}^{(\ast)'}$ with 
$\lvert d_{\bullet}^{(\ast)} \rvert = \lvert d_{\bullet}^{(\ast)'}\rvert$, 
the subcategories 
$ \cC(d_{\bullet}^{(\ast)})$
and $\cC(d_{\bullet}^{(\ast)'})$ are orthogonal. 
Indeed in this case 
there exist $j, j'$ such that 
$d_{\bullet}^{(j)} \prec d_{\bullet}^{(j)'}$
and $d_{\bullet}^{(j)} \succ d_{\bullet}^{(j')'}$
so the orthogonality follows from 
the last statement of Theorem~\ref{thm:SOD}. 
Similarly if $\lvert d_{\bullet}^{(\ast)'} \rvert \succ \lvert d_{\bullet}^{(\ast)}\rvert$, 
then there is $j$ such that 
$\lvert d_{\bullet}^{(j)'} \rvert \succ \lvert d_{\bullet}^{(j)}\rvert$
so we have 
$\Hom(\cC(d_{\bullet}^{(\ast)'}), \cC(d_{\bullet}^{(\ast)}))=0$. 
Therefore the theorem holds. 
	\end{proof}

	\subsection{Some versions for categories of factorizations}\label{subsec:someversion}
	We will use some variants of Theorem~\ref{thm:SOD:mult} for 
	categories of factorizations on some formal completions.  
	Let 
	\begin{align*}
			\mM_{Q^{(m)}}(d^{(\ast)}) \to M_{Q^{(m)}}(d^{(\ast)})
		\end{align*}
	be the good moduli space. 
	Let 
	$R$ be a complete local $\mathbb{C}$-algebra 
	with closed point $0 \in \Spec R$. 
	We denote by $\widehat{M}_{Q^{(m)}}(d^{(\ast)})_R$ the 
	following formal completion of $M_{Q^{(m)}}(d^{(\ast)}) \times \Spec R$
	\begin{align*}
		\widehat{M}_{Q^{(m)}}(d^{(\ast)})_R
		 \cneq \Spec \widehat{\oO}_{M_{Q^{(m)}}(d^{(\ast)}) \times \Spec R, (0, 0)},
		\end{align*}
	and take the following formal fibers
	\begin{align}\label{dia:formal}
		\xymatrix{
		\widehat{\gG}^{(m)}_{a, b}(d^{(\ast)})_R \ar[r] \ar[d]  \diasquare
		& \gG_{a, b}^{(m)}(d^{(\ast)}) \times \Spec R \ar[d] \\
		\widehat{\mM}_{Q^{(m)}}(d^{(\ast)})_R \ar[r] \ar[d] \diasquare
		& \mM_{Q^{(m)}}(d^{(\ast)}) \times \Spec R \ar[d] \\
		\widehat{M}_{Q^{(m)}}(d^{(\ast)})_R \ar[r]
		& M_{Q^{(m)}}(d^{(\ast)}) \times \Spec R. 
	}
		\end{align}
	Here the upper right vertical arrow is the projection.  
	We consider an auxiliary $\C$-action on 
	$\gG_{a, b}^{(m)}(d^{(\ast)})$ acting on 
	maps from $\infty$ to each vertex $\{1, \ldots, m\}$ with weight one, 
	and acts on $\Spec R$ trivially. 
	Then it induces the $\C$-action on 
	$\widehat{\gG}_{a, b}^{(m)}(d^{(\ast)})_R$. 
	We take a regular function 
	\begin{align}\label{funct:formal}
		w \colon \widehat{\gG}_{a, b}^{(m)}(d^{(\ast)})_R
		\to \mathbb{A}^1
		\end{align}
	which is of weight one with respect to the above 
	$\C$-action. 
	We consider the triangulated 
	category of $\C$-equivariant 
	factorizations $\MF^{\C}(\widehat{\gG}_{a, b}^{(m)}(d^{(\ast)})_R, w)$. 
	
	The diagram (\ref{dia:Glm}) 
	extends to the diagram 
	\begin{align}\label{dia:formal0}
	\xymatrix{\gG_{a, b}^{(m)}(d^{(\ast)})^{\lambda \ge 0} \times \Spec R
		 \ar[r] \ar[d]  
		& \gG_{a, b}^{(m)}(d^{(\ast)}) \times \Spec R \ar[dd] \\
		\mM_{Q^{(m)}}(d_1^{(\ast)})
		\times \gG_{a, b}^{(m)}(d_2^{(\ast)})\times \Spec R
		\ar[d] 
		&  \\
		M_{Q^{(m)}}(d_1^{(\ast)}) \times
		M_{Q^{(m)}}(d_2^{(\ast)}) \times \Spec R
		 \ar[r]^-{\oplus}
		& M_{Q^{(m)}}(d^{(\ast)}) \times \Spec R. 
	}
\end{align}
Here the right vertical arrow and the left bottom vertical arrow
are compositions in the right vertical arrow in (\ref{dia:formal}). 	
By taking pull-back via the bottom horizontal arrow in (\ref{dia:formal}), 
and taking account of the function (\ref{funct:formal}), 
we obtain the diagram	
	\begin{align}\label{dia:formal2}
	\xymatrix{\widehat{\gG}_{a, b}^{(m)}(d^{(\ast)})^{\lambda \ge 0}_R
		\ar[r] \ar[d]   
		& \widehat{\gG}_{a, b}^{(m)}(d^{(\ast)})_R \ar[dd] \ar[rd]^-{w} &\\
		\widehat{\mM}_{Q^{(m)}}(d_1^{(\ast)})_R \widehat{\times}_R
		\widehat{\gG}_{a, b}^{(m)}(d_2^{(\ast)})_R 
		\ar[d] \ar[rr]^{(0, w')}
		& & \mathbb{A}^1 \\
		\widehat{M}_{Q^{(m)}}(d_1^{(\ast)})_R \widehat{\times}_R
		\widehat{M}_{Q^{(m)}}(d_2^{(\ast)})_R
		\ar[r]^-{\oplus}
		& \widehat{M}_{Q^{(m)}}(d^{(\ast)})_R. &
	}
\end{align}
	Here in the above diagram, the function 
	$w$ descends to a function of the form 
	$(0, w')$ in the 
	middle horizontal arrow by the $\C$-weight one condition of $w$. 
	By the abuse of notation, we also denote $w'$ by $w$. 
	From the above diagram, the Hall product (\ref{caHall:mult}) induces 
	the one for formal completions
	\begin{align}\label{Hprod:hat}
		\ast \colon 
		D^b(\widehat{\mM}_{Q^{(m)}}(d_1^{(\ast)})_R) \widehat{\boxtimes}_R
	\MF^{\C}(\widehat{\gG}_{a, b}^{(m)}(d_2^{(\ast)})_R, w)	
	\to \MF^{\C}(\widehat{\gG}_{a, b}^{(m)}(d^{(\ast)})_R, w). 
		\end{align}
	
	The window subcategories 
\begin{align}\label{abuse}
	\widehat{\mathbb{W}}(d^{(\ast)}) \subset 
	D^b(\widehat{\mM}_{Q^{(m)}}(d^{(\ast)})_R), \ 
	\widehat{\mathbb{W}}_c(d^{(\ast)}) \subset \MF^{\C}(\widehat{\gG}_{a, b}^{(m)}(d^{(\ast)})_R, w)
\end{align}
are also defined similarly to (\ref{box:product}): 
when $m=1$ and $d^{(\ast)}=d$, they 
are the smallest thick 
triangulated subcategories which contain
$\oO_{\widehat{\mM}_Q(d)_R}$, 
factorizations with entries 
direct sums of
$V(\chi) \otimes \oO_{B\C}(j)\otimes\oO$
for $\chi \in \mathbb{B}_c(d)$ and $j\in \mathbb{Z}$, respectively. 
Here $\oO_{B\C}(j)$ is the one dimensional $\C$-representation with weight 
$j$. For $m>1$, they are defined to be the box-products of 
window subcategories of
each factors as in (\ref{box:product}).
	The result of Theorem~\ref{thm:SOD:mult} immediately 
implies the following variant of it 
(for example see the argument of~\cite[Corollary~4.22]{Toconi}):
\begin{thm}\label{thm:SOD2}
	We take $c>b$ with $l \cneq c-b>0$. 
	For each $d_{\bullet}=(d_1, \ldots, d_l) \in \mathbb{Z}_{\ge 0}^l$, 
the functors (\ref{Hprod:hat}) induce the fully-faithful functor 
	\begin{align*}
		\ast \colon 	\bigoplus_{\begin{subarray}{c}
				(d_1^{(\ast)}, \ldots, d_r^{(\ast)}) \\
				\lvert d_i^{(\ast)} \rvert=d_i
		\end{subarray}}
		\widehat{\mathbb{W}}(d_1^{(\ast)}) \widehat{\boxtimes}_R
		(\widehat{\mathbb{W}}(d_2^{(\ast)}) \otimes \chi_0)
		\widehat{\boxtimes}_R \cdots \widehat{\boxtimes}_R (\widehat{\mathbb{W}}(d_l^{(\ast)}) \otimes \chi_0^{l-1})
		&\widehat{\boxtimes}_R (\widehat{\mathbb{W}}_b(d-\sum_{i=1}^l d_i) \otimes \chi_0^l) \\
		&\to \widehat{\mathbb{W}}_c(d^{(\ast)})
	\end{align*}
	such that, 
	by setting 
	$\widehat{\cC}(d_{\bullet}) \subset 
	\MF^{\C}(\widehat{\gG}_{a, b}^{(m)}(d^{(\ast)})_R, w)$ to be the 
	essential image of the above functor, 
	we have the semiorthogonal decomposition 
	\begin{align*}
		\widehat{\mathbb{W}}_c(d^{(\ast)})=
		\langle \widehat{\cC}(d_{\bullet}) : 
		d_{\bullet} \in \mathbb{Z}_{\ge 0}^l \rangle. 
	\end{align*}
	Here $\Hom(\widehat{\cC}(d'_{\bullet}), \widehat{\cC}(d_{\bullet}))=0$
	for $d_{\bullet}' \succ d_{\bullet}$. 
\end{thm}

\section{Quot formula of relative dimension one}
\subsection{Relative Quot schemes}
Let $S$ be a smooth quasi-projective scheme 
and 
\begin{align*}
	\pi \colon \cC \to S
	\end{align*} 
be a smooth projective morphism of relative dimension one. 
The stack of $S$-relative zero-dimensional sheaves of length $d$ is 
given by the 2-functor
\begin{align*}
	\mM_{\cC/S}(d) \colon (Sch/S) \to (Groupoid)
	\end{align*}
which sends $T \to S$ to the groupoid of 
$T$-flat sheaves $\pP \in \Coh(\cC_T)$
such that for any $t \in T$, the object
$\pP_t \in \Coh(\cC_t)$ is zero-dimensional of length $d$. 
The stack $\mM_{\cC/S}(d)$ is a smooth Artin stack 
such that the structure morphism 
\begin{align*}
	\mM_{\cC/S}(d) \to S
	\end{align*}
is smooth whose fiber at $s \in S$ is the 
stack of zero-dimensional sheaves of length $d$ on the smooth curve 
$\cC_s$. 

For $\eE \in \Coh(\cC)$ with $r \cneq \rank(\eE) \ge 0$, 
let
$\mM_{\cC/S}(\eE, d)$ be the 2-functor 
\begin{align}\label{2funct}
	\mM_{\cC/S}(\eE, d) \colon 
	(Sch/S) \to (Groupoid)
\end{align}
by sending $T$ to the groupoid of pairs 
$(\pP, u)$ where $\pP$ is a $T$-valued point 
of $\mM_{\cC/S}(d)$ and $u \colon \eE_T \to \pP$ is a morphism. 
The set of isomorphisms is given by commutative diagrams 
\begin{align*}
	\xymatrix{
\eE_T \ar[r]^-{u} \ar@{=}[d] & \pP \ar[d]^-{\cong} \\
\eE_T \ar[r]^-{u'} & \pP'. 	
}
	\end{align*}
The 2-functor (\ref{2funct}) is an Artin stack
with morphisms
\begin{align}\label{factor:E0}
	\mM_{\cC/S}(\eE, d) \to \mM_{\cC/S}(d) \to S 
	\end{align}
where the first arrow is forgetting $u$. 

The stack $\mM_{\cC/S}(\eE, d)$ contains an open substack
\begin{align*}
	\Quot_{\cC/S}(\eE, d) \subset \mM_{\cC/S}(\eE, d)
	\end{align*}
corresponding to $(\pP, u)$ such that $u \colon \eE_T \to \pP$ is surjective. 
The substack $\Quot_{\cC/S}(\eE, d)$ is the $S$-relative Grothendieck 
Quot scheme of $\eE$, and it is a projective scheme over $S$. 
The fiber at $s \in S$ is the 
Quot scheme which parametrizes 
surjections $\eE_s \twoheadrightarrow \qQ$
where $\qQ \in \Coh(\cC_s)$ is zero-dimensional of length $d$. 

If $\eE$ is locally free, then 
the first arrow in (\ref{factor:E0})
is the total space of a vector bundle over $\mM_{\cC/S}(d)$. 
Indeed let 
\begin{align}\label{univ:Q}
	\qQ \in \Coh(\cC \times_S \mM_{\cC/S}(d))
\end{align} be 
the universal zero-dimensional sheaf, and $p_1, p_2$ be the projections 
from $\cC \times_S \mM_{\cC/S}(d)$ onto the corresponding factors. 
Then 
$p_{2\ast}(\hH om(p_1^{\ast}\eE, \qQ))$ is a locally free sheaf on $\mM_{\cC/S}(d)$, 
and 
we have 
\begin{align*}
	\mM_{\cC/S}(\eE, d)=\mathrm{Tot}(p_{2\ast}(\hH om(p_1^{\ast}\eE, \qQ)))
	=\Spec_{\mM_{\cC/S}(d)} \mathrm{Sym}(p_{2\ast}(\eE^{\vee} \boxtimes \qQ)^{\vee}). 
\end{align*}
In particular in this case $\mM_{\cC/S}(\eE, d)$ is smooth of relative dimension 
$rd$. 

Let $W$ be a $d$-dimensional vector space, 
and $\GL_S(W) \cneq \GL(W) \times S \to S$
be the group scheme over $S$. 
Let 
\begin{align*}
	\Quot_{\cC/S}^{\circ}(W \otimes \oO_{\cC}, d) \subset 
	\Quot_{\cC/S}(W \otimes \oO_{\cC}, d)
	\end{align*}
be the open subscheme whose $T$-valued points 
correspond to $u \colon W \otimes \oO_{\cC_T} \twoheadrightarrow \pP$
such that the induced morphism 
$W \otimes \oO_T \to \pi_{T\ast}\pP$ is an isomorphism. 
Since any zero-dimensional sheaf is globally generated, 
the first morphism in (\ref{factor:E0})
induces the isomorphism over $S$
\begin{align}\label{isom:QuotW}
	[\Quot_{\cC/S}^{\circ}(W \otimes \oO_{\cC}, d)/\GL_S(W)]
	\stackrel{\cong}{\to} \mM_{\cC/S}(d). 
	\end{align}

\subsection{Derived structures of relative Quot schemes}
Suppose that $\eE$ has homological dimension less than 
or equal to one. 
Then there is a locally free resolution 
\begin{align}\label{loc:free}
	0 \to \eE^{-1} \stackrel{\phi}{\to} \eE^0 \to \eE \to 0
\end{align}
such that 
\begin{align}\label{bundle:phi}
\hH \cneq \mathcal{E}xt^1_{\oO_{\cC}}(\eE, \oO_{\cC})
	=\mathrm{Cok}(\phi^{\vee} \colon \eE_0 \to \eE_1). 
\end{align}
Here we have set $\mathcal{E}_0 \cneq (\mathcal{E}^0)^{\vee}$
and $\mathcal{E}_{1} \cneq (\mathcal{E}^{-1})^{\vee}$. 
The morphism $\phi$ induces the morphism of vector 
bundles on $\mM_{\cC/S}(d)$
\begin{align*}
	\phi \colon 
	\mM_{\cC/S}(\eE^0, d)  \to \mM_{\cC/S}(\eE^{-1}, d)
	\end{align*}
which, on $T$-valued points, is defined by 
\begin{align}\label{mor:phi}
	(u \colon \eE_T^0 \to \pP) \mapsto (u \circ \phi \colon 
	\eE_T^{-1} \stackrel{\phi}{\to} \eE_T^0 \stackrel{u}{\to} \pP). 
\end{align}
We define the derived stack $\mathbf{M}_{\cC/S}(\eE, d)$
by the derived Cartesian square
\begin{align}\label{dia:MC}
	\xymatrix{
\mathbf{M}_{\cC/S}(\eE, d) \ar[r] \ar[d] \diasquare & \mM_{\cC/S}(d) \ar[d]^-{0} \\
\mM_{\cC/S}(\eE^0, d) \ar[r]^-{\phi} & \mM_{\cC/S}(\eE^{-1}, d). 
}
	\end{align}
Here the right vertical arrow is the zero section
of the vector bundle $\mM_{\cC/S}(\eE^{-1}, d) \to \mM_{\cC/S}(d)$. 
The classical truncation of $\mathbf{M}_{\cC/S}(\eE, d)$ is isomorphic 
to $\mM_{\cC/S}(\eE, d)$. 

The surjection $\mathcal{E}^0 \twoheadrightarrow \eE$
induces the closed immersion 
\begin{align}\label{quot:closed}
	\Quot_{\cC/S, d}(\eE) \hookrightarrow 
	\Quot_{\cC/S, d}(\mathcal{E}^0),
\end{align}
where the target is an open substack of $\mM_{\cC/S}(\eE^0, d)$. 
We define the quasi-smooth 
derived scheme $\mathbf{Quot}_{\cC/S}(\eE, d)$ over $S$
by the Cartesian square 
\begin{align}\label{dia:MC2}
	\xymatrix{
	\mathbf{Quot}_{\cC/S}(\eE, d) \ar@<-0.3ex>@{^{(}->}[r] \ar[d] \diasquare & \mathbf{M}_{\cC/S}(\eE, d) \ar[d] \\
	\Quot_{\cC/S}(\eE^0, d) \ar@<-0.3ex>@{^{(}->}[r] & \mM_{\cC/S}(\eE^{0}, d). 
}
		\end{align}
	Here horizontal arrows are open immersions. 
Note that $\mathbf{Quot}_{\cC/S}(\eE, d)$ is a derived open substack of 
$\mathbf{M}_{\cC/S}(\eE, d)$,
with virtual dimension $\dim S+rd$, and
 whose classical truncation is 
$\Quot_{\cC/S}(\eE, d)$. 

By taking the dual of the sequence (\ref{loc:free}), 
we obtain the exact sequence 
\begin{align*}
	\eE_0 \stackrel{\phi^{\vee}}{\to} \eE_1 \to \hH \to 0. 
	\end{align*}
Similarly to (\ref{dia:MC}) and (\ref{dia:MC2}), we 
define $\mathbf{M}_{\cC/S}(\hH, d)$ 
and $\mathbf{Quot}_{\cC/S}(\hH, d)$ by the 
derived Cartesian squares 
\begin{align}\label{dia:dual}
	\xymatrix{
\mathbf{Quot}_{\cC/S}(\hH, d) \ar@<-0.3ex>@{^{(}->}[r] \ar[d] \diasquare & 
		\mathbf{M}_{\cC/S}(\hH, d) \ar[r] \ar[d] \diasquare & \mM_{\cC/S}(d) \ar[d]^-{0} \\
		\Quot_{\cC/S}(\eE_1, d) \ar@<-0.3ex>@{^{(}->}[r] & 
		\mM_{\cC/S}(\eE_1, d) \ar[r]^-{\phi^{\vee}} & \mM_{\cC/S}(\eE_0, d). 
	}
\end{align}
Here the bottom right arrow is defined similarly to (\ref{mor:phi})
from $\phi^{\vee} \colon \eE_0 \to \eE_1$. 
The derived stack $\mathbf{Quot}_{\cC/S}(\hH, d)$
is a derived open substack of 
$\mathbf{M}_{\cC/S}(\hH, d)$
 with virtual dimension $\dim S-rd$, and its classical 
truncation is $\Quot_{\cC/S}(\hH, d)$. 

\subsection{$(-1)$-shifted cotangent construction}
By applying the construction in (\ref{funct:Fw})
for the morphism of vector bundles (\ref{bundle:phi}), 
we obtain the stack $\nN_{\cC/S}(\eE^{\bullet}, d)$ 
with a regular function $w$
\begin{align*}
	\nN_{\cC/S}(\eE^{\bullet}, d) \cneq \mM_{\cC/S}(\eE^0, d) \times_{\mM_{\cC/S}(d)}
	\mM_{\cC/S}(\eE^{-1}, d)^{\vee} \stackrel{w}{\to} \mathbb{A}^1. 
	\end{align*}
By the Grothendieck duality, the $T$-valued 
points of the stack $\nN_{\cC/S}(\eE^{\bullet}, d)$
consist of 
\begin{align}\label{Tvalue}
	(\eE_T^0 \stackrel{u}{\to} \pP \stackrel{v}{\to} \eE_T^{-1} \otimes \omega_{\cC_T/T}[1])
	\end{align}
where $\pP$ is a $T$-valued point of $\mM_{\cC/S}(d)$. 
For an affine $T$, the function $w$ on the 
$T$-valued point (\ref{Tvalue}) 
is given by 
\begin{align}\label{tr:1}
	\mathrm{Tr}(\eE_T^0 \stackrel{u}{\to} \pP \stackrel{v}{\to} \eE_T^{-1} \otimes \omega_{\cC_T/T}[1]
	\stackrel{\phi}{\to} \eE_T^0 \otimes \omega_{\cC_T/T}[1])
	\in H^1(\cC_T, \omega_{\cC_T/T}) =\oO_{T}. 
	\end{align}
By the construction, we have the isomorphism (see Subsection~\ref{subsec:Koszul})
\begin{align*}
	t_0(\Omega_{\mathbf{M}_{\cC/S}(\eE, d)}[-1]) \cong 
	\Crit(w) \subset \nN_{\cC/S}(\eE^{\bullet}, d). 
	\end{align*}

By applying the similar construction for the 
bottom right arrow in (\ref{dia:dual}), 
we obtain the stack 
$\nN_{\cC/S}(\eE_{\bullet}, d)$ 
with a regular function $w^{\vee}$ on it
\begin{align*}
	\nN_{\cC/S}(\eE_{\bullet}, d) \cneq \mM_{\cC/S}(\eE_1, d) \times_{\mM_{\cC/S}(d)}
	\mM_{\cC/S}(\eE_0, d)^{\vee} \stackrel{w^{\vee}}{\to} \mathbb{A}^1. 
	\end{align*}
The $T$-valued 
points of the stack $\nN_{\cC/S}(\eE_{\bullet}, d)$
consist of 
\begin{align}\label{Tvalue2}
	((\eE_1)_T \stackrel{u'}{\to} \pP \stackrel{v'}{\to} (\eE_0)_T \otimes \omega_{\cC_T/T}[1])
\end{align}
where $\pP$ is a $T$-valued point of $\mM_{\cC/S}(d)$. 
For an affine $T$, the function $w^{\vee}$ on the 
$T$-valued point (\ref{Tvalue2}) 
is given by 
\begin{align}\label{tr:2}
	\mathrm{Tr}((\eE_1)_T \stackrel{u'}{\to} \pP \stackrel{v'}{\to} (\eE_0)_T \otimes \omega_{\cC_T/T}[1]
	\stackrel{\phi^{\vee}}{\to} (\eE_1)_T \otimes \omega_{\cC_T/T}[1])
	\in H^1(\cC_T, \omega_{\cC_T/T}) =\oO_{T}. 
\end{align}
We also have the isomorphism 
\begin{align*}
	t_0(\Omega_{\mathbf{M}_{\cC/S}(\hH, d)}[-1]) \cong 
	\Crit(w^{\vee}) \subset \nN_{\cC/S}(\eE_{\bullet}, d). 
\end{align*}

For a $T$-valued point $\pP \in \Coh(\cC_T)$ of $\mM_{\cC/S}(d)$, 
we define 
\begin{align*}
	\pP^{\vee} \cneq \dR \hH om_{\cC_T}(\pP, \omega_{\cC_T/T}[1])
	=\eE xt^1_{\cC_T}(\pP, \omega_{\cC_T/T}) \in \Coh(\cC_T). 
	\end{align*}
The above object also gives a $T$-valued point of $\mM_{\cC/S}(d)$, 
so we obtain the involution isomorphism 
\begin{align}\label{M:invo}
	\mathbb{D} \colon \mM_{\cC/S}(d) \stackrel{\cong}{\to} \mM_{\cC/S}(d), \ 
	\pP \mapsto \pP^{\vee}.
	\end{align}

\begin{lem}\label{lem:inv}
	There is an isomorphism 
	\begin{align}\label{isom:inv2}
		\mathbb{D} \colon 
		\nN_{\cC/S}(\eE^{\bullet}, d) \stackrel{\cong}{\to}\nN_{\cC/S}(\eE_{\bullet}, d)
		\end{align}
	such that the following diagram commutes: 
	\begin{align}\label{dia:commute}
		\xymatrix{
& \mathbb{A}^1 & \\
\nN_{\cC/S}(\eE^{\bullet}, d) \ar[ru]^-{w} \ar[rr]^-{\cong}_-{\mathbb{D}} \ar[d]
 & &  \nN_{\cC/S}(\eE_{\bullet}, d) \ar[d]	\ar[lu]_-{w^{\vee}} \\
\mM_{\cC/S}(d) \ar[rr]^-{\cong}_-{\mathbb{D}} & & \mM_{\cC/S}(d). 	
}
		\end{align}
	Here the vertical arrows are projections. 
	\end{lem}
\begin{proof}
	By applying $\dR \hH om_{\cC_T/T}(-, \omega_{\cC_T/T}[1])$
	to (\ref{Tvalue}), we obtain
	\begin{align}\label{Tvalue3}
		((\eE_T)_1  \stackrel{v^{\vee}}{\to} \pP^{\vee} \stackrel{u^{\vee}}{\to}
		(\eE_T)_0 \otimes \omega_{\cC_T/T}[1]). 
		\end{align}
	The isomorphism (\ref{isom:inv2}) is obtained by assigning (\ref{Tvalue})
	with (\ref{Tvalue3}). 
	Then the diagram (\ref{dia:commute}) obviously commutes. 
	\end{proof}

\subsection{$\Theta$-stratification}
The stack $\mM_{\cC/S}(d)$ is the $S$-relative moduli stack of 
zero-dimensional semistable sheaves on $\cC$, 
so it admits a good moduli space
(which is nothing but the symmetric product)
\begin{align}\label{pi:MCS}
	\pi_{\mM} \colon 
	\mM_{\cC/S}(d) \to M_{\cC/S}(d) =\mathrm{Sym}^d(\cC/S). 
	\end{align}
A point $p \in \mathrm{Sym}^d(\cC/S)$ corresponds to 
a point $s \in S$ and an effective divisor on $\cC_s$ of degree $d$
\begin{align}\label{point:p}
	p=\sum_{j=1}^m d^{(j)} p^{(j)}, 	
	\end{align}
where $p^{(1)}, \ldots, p^{(m)}$ are distinct points in $\cC_s$
and $d^{(j)} \ge 0$ with 
$d^{(1)}+\cdots+d^{(m)}=d$. 
Let $R$ be the complete local $\mathbb{C}$-algebra 
$R=\widehat{\oO}_{S, s}$. 
Formally locally at $p \in M_{\cC/S}(d)$, the stack $\mM_{\cC/S}(d)$ is 
described by 
the following commutative diagram 
\begin{align}\label{dia:formal10}
	\xymatrix{
\widehat{\mM}_{Q^{(m)}}(d^{(\ast)})_R \ar[r]^-{\cong} \ar[d] 
\ar@/^1.5pc/[rr]^-{\iota_p}& \widehat{\mM}_{\cC/S}(d)_p \ar[r] \ar[d]^-{\pi_p} 
\diasquare
& \mM_{\cC/S}(d) \ar[d]^-{\pi_{\mM}} \\
\widehat{M}_{Q^{(m)}}(d^{(\ast)})_R	\ar[r]^-{\cong} & \widehat{M}_{\cC/S}(d)_p \ar[r] & 
M_{\cC/S}(d).
}
	\end{align}
Here the middle vertical arrow is the formal fiber at $p$, 
and we have used the notation in Subsection~\ref{subsec:someversion} for the left vertical 
arrow.

The morphism 
$\nN_{\cC/S}(\eE^{\bullet}, d) \to \mM_{\cC/S}(d)$ is an affine morphism, 
so it also admits a good moduli space $N_{\cC/S}(\eE^{\bullet}, d)$ 
together with a commutative diagram 
\begin{align}\label{dia:gmoduli}
	\xymatrix{
\nN_{\cC/S}(\eE^{\bullet}, d) \ar[r] \ar[d] \ar[rd]^-{g} & \mM_{\cC/S}(d) \ar[d] \\
N_{\cC/S}(\eE^{\bullet}, d) \ar[r] & M_{\cC/S}(d). 
}
\end{align}
Indeed $N_{\cC/S}(\eE^{\bullet}, d)=\Spec g_{\ast}\oO_{\nN_{\cC/S}(\eE^{\bullet}, d)}$, where 
$g$ is the clockwise composition in the diagram (\ref{dia:gmoduli}). 

For the universal sheaf $\qQ$ in (\ref{univ:Q}), 
we set 
\begin{align*}
	\lL \cneq \det(p_{2\ast}\qQ) \in \Pic(\mM_{\cC/S}(d)), \ 
	b \cneq \ch_2(p_{2\ast}\qQ) \in H^4(\mM_{\cC/S}(d), \mathbb{Q}). 
\end{align*}
We also regard them as elements of 
$\Pic(\nN_{\cC/S}(\eE^{\bullet}, d))$, 
$H^4(\nN_{\cC/S}(\eE^{\bullet}, d), \mathbb{Q})$ by pulling back them via 
$\nN_{\cC/S}(\eE^{\bullet}, d) \to \mM_{\cC/S}(d)$. 
\begin{lem}
	The element $b \in H^4(\nN_{\cC/S}(\eE^{\bullet}, d), \mathbb{Q})$ is positive definite. 
	\end{lem}
\begin{proof}
	Let $f \colon B\C \to \nN_{\cC/S}(\eE^{\bullet}, d)$ be a non-degenerate 
	morphism. 
	It corresponds to a point $s \in S$ and a diagram 
	\begin{align*}
		(\eE^0_s \to \pP_0 \to \eE^{-1}_s \otimes \omega_{\cC_s}[1])
		\oplus \bigoplus_{j\neq 0} (0 \to \pP_j \to 0)
		\end{align*}
	where $\pP_j$ are zero-dimensional sheaves on $\cC_s$ and have
	$\C$-weight $j$. 
	As $f$ is non-degenerate, we have $\pP_j \neq 0$ for some $j \neq 0$. 
	Then we have 
	\begin{align*}
		q^{-2}f^{\ast}b=
		\sum_{j} \frac{1}{2}j^2 \cdot \chi(\pP_j)>0. 
		\end{align*}
	\end{proof}

By the above lemma, there exist associated $\Theta$-stratifications
with respect to $(\lL^{\pm 1}, b)$ 
\begin{align}\label{theta:N}
	\nN_{\cC/S}(\eE^{\bullet}, d)=
	\sS_0^{\pm} \sqcup \sS_1^{\pm} \sqcup \cdots \sqcup \sS_{N^{\pm}}^{\pm} \sqcup 
	\nN_{\cC/S}(\eE^{\bullet}, d)^{\lL^{\pm 1}\sss}. 	
	\end{align}
The stack $\nN_{\cC/S}(\eE^{\bullet}, d)$ and its $\Theta$-stratifications 
are described formally locally on $M_{\cC/S}(d)$ in the following way. 
For 
$p \in M_{\cC/S}(d)$, 
we write it as (\ref{point:p}), 
and set $R=\widehat{\oO}_{S, s}$. We have the commutative diagram 
\begin{align}\label{dia:formal20}
	\xymatrix{
		\widehat{\gG}_{a, b}^{(m)}(d^{(\ast)})_R 
		\ar[r]^-{\cong} \ar[d] \ar@/^1.5pc/[rr]^-{\iota_p}
		& \widehat{\nN}_{\cC/S}(\eE^{\bullet}, d)_p \ar[r] \ar[d] 
		\diasquare
		& \nN_{\cC/S}(\eE^{\bullet}, d) \ar[d]^-{g} \\
\widehat{M}_{Q^{(m)}}(d^{(\ast)})_R		\ar[r]^-{\cong} & 
	\widehat{M}_{\cC/S}(d)_p \ar[r] & 
	M_{\cC/S}(d).
	}
\end{align}
Here $a \cneq \rank \eE^0$, $b \cneq \rank \eE^{-1}$, 
and we have used the notation in Subsection~\ref{subsec:someversion} for the left vertical 
arrow. 
Since $(\lL, b)$ pulls back to $(\chi_0, \lvert \ast \rvert)$
under the top arrows in (\ref{dia:formal20}), 
the $\Theta$-stratification pulls back 
to the KN
stratifications in (\ref{KN:Gm}). 
 In particular, we have $N^{\pm}=d-1$.

 \subsection{Window subcategories}
 For the good moduli space morphism (\ref{pi:MCS}), 
 since $M_{\cC/S}(d)$ is smooth the functor 
 \begin{align}\label{funct:piast}
 	\pi_{\mM}^{\ast} \colon D^b(M_{\cC/S}(d)) \to D^b(\mM_{\cC/S}(d))
 	\end{align}
 is well-defined and it is fully-faithful by the definition of 
 good moduli space. 
 We define the triangulated subcategory 
 \begin{align}\label{window:M1}
 	\mathbb{W}_{\rm{glob}}(d) \subset D^b(\mM_{\cC/S}(d))
 	\end{align}
 to be the essential image of the functor (\ref{funct:piast}). 
 \begin{lem}\label{lem:loc1}
 	An object $\eE \in D^b(\mM_{\cC/S}(d))$ lies in 
 	$\mathbb{W}_{\rm{glob}}(d)$ if and only if for any $p \in M_{\cC/S}(d)$
 	as in (\ref{point:p}),
 	we have $\iota_p^{\ast}\eE \in \widehat{\mathbb{W}}(d^{(\ast)})$.
 	Here $\iota_p$ is given in (\ref{dia:formal10}) and 
 	$\widehat{\mathbb{W}}(d^{(\ast)})$ is given in (\ref{abuse}).  
 	\end{lem}
 \begin{proof}
An object $\eE \in D^b(\mM_{\cC/S}(d))$ lies in 
$\mathbb{W}_{\rm{glob}}(d)$ if and only if the adjunction morphism 
\begin{align*}
	\pi_{\mM}^{\ast} \pi_{\mM\ast}\eE \to \eE
	\end{align*}
is an isomorphism. 
Let $\fF$ be the cone of the above morphism. 
Then the last condition is equivalent to that $\fF=0$. 
This property is formally local on $M_{\cC/S}(d)$, 
so it is equivalent to that $\iota_p^{\ast}\fF=0$ 
for any $p \in M_{\cC/S}(d)$. 
By the base change, it is also equivalent to that 
\begin{align*}
	\pi_p^{\ast}\pi_{p\ast}(\iota_p^{\ast}\eE) \to \iota_p^{\ast}\eE
	\end{align*}
is an isomorphism, which is equivalent to 
that $\iota_p^{\ast}\eE \in \widehat{\mathbb{W}}(d^{(\ast)})$.  	
 	\end{proof}
 
 We take the $\C$-action on $\nN_{\cC/S}(\eE^{\bullet}, d)$
 by the weight one action on the factor 
 $\mM_{\cC/S}(\eE^0, d)$, i.e $t \in \C$
 acts on (\ref{Tvalue}) 
 by 
 \begin{align*}
 		(\eE_T^0 \stackrel{tu}{\to} \pP \stackrel{v}{\to} \eE_T^{-1} \otimes \omega_{\cC_T/T}[1]). 
 	\end{align*}
 The function (\ref{tr:1}) is of weight one with respect to the 
 above $\C$-action. 
 So we have the triangulated category of $\C$-equivariant 
 factorizations of $w$
 \begin{align*}
 	\MF^{\C}(\nN_{\cC/S}(\eE^{\bullet}, d), w). 
 	\end{align*}
Let us consider $\Theta$-stratifications (\ref{theta:N}). 
For each $1\le i\le N^{\pm}$, let 
us take $m_{i}^{\pm} \in \mathbb{R}$. 
Similarly to (\ref{window:m}), 
the window subcategory 
\begin{align*}
\mathbb{W}_{m_{\bullet}^{\pm}}(d) \subset 
		\MF^{\C}(\nN_{\cC/S}(\eE^{\bullet}, d), w)
	\end{align*}
is defined to be consisting of objects $\pP$
so that for each $1\le i \le N^{\pm}$ and 
center $\zZ_i^{\pm} \subset \sS_i^{\pm}$, we have
the weight condition with respect to the canonical 
$\C$-stabilizer groups in $\zZ_i^{\pm}$: 
\begin{align}\label{cond:weight}
	\wt(\pP|_{\zZ_i^{\pm}})
	\subset [m_i^{\pm}, m_i^{\pm}+\eta_i^{\pm})
	\end{align}
Here $\eta_{i}^{\pm} \in \mathbb{Z}$ are defined as in (\ref{etai}), 
i.e. the weight of the conormal bundle of 
$\sS_i^{\pm}$ inside $\nN_{\cC/S}(\eE^{\bullet}, d)$
restricted to $\zZ_i^{\pm}$. 

\begin{prop}\label{prop:window:eq}
	There exist equivalences 
	\begin{align}\label{weq:N}
	\mathbb{W}_{m_{\bullet}^+}(d) \simeq D^b(\mathbf{Quot}(\eE, d)), \
\mathbb{W}_{m_{\bullet}^-}(d) \simeq D^b(\mathbf{Quot}(\hH, d)). 	
		\end{align}
	\end{prop}
\begin{proof}
	A version of window theorem 
	implies that the composition functors 
	\begin{align}\label{equiv:Wquot}
		\mathbb{W}_{m_{\bullet}^{\pm}}(d)
		\hookrightarrow 	\MF^{\C}(\nN_{\cC/S}(\eE^{\bullet}, d), w)
		\twoheadrightarrow 	\MF^{\C}(\nN_{\cC/S}
		(\eE^{\bullet}, d)^{\lL^{\pm 1}\sss}, w)
		\end{align}
	are equivalences
	(see~\cite[Prop 2.4.2]{HalpK3} for perfect complexes, 
	and then use the argument of~\cite[Proposition~5.5]{MR3327537} to deduce 
	the result for factorization categories). 
	From the formal local description of $\Theta$-stratifications 
by the diagram (\ref{dia:formal20}), 
it follows that (\ref{Tvalue}) is a $T$-valued point of 
$\nN_{\cC/S}(\eE^{\bullet}, d)^{\lL\sss}$ if and only if 
$u \colon \eE_T^0 \to \pP$ is surjective. 
Therefore
from Theorem~\ref{thm:koszul}, 
the last category in (\ref{equiv:Wquot}) for $\lL$-semistable 
part is equivalent to 
$D^b(\mathbf{Quot}(\eE, d))$, 
so 
the first equivalence in (\ref{weq:N})
follows. 

Under the isomorphism (\ref{M:invo}), 
the element line bundle $\lL$
pulls back to $\lL^{-1}$. 
Therefore from the diagram (\ref{dia:commute}),
the isomorphism (\ref{isom:inv2}) restricts to the isomorphism  
\begin{align*}
	\mathbb{D} \colon 
	\nN_{\cC/S}(\eE^{\bullet}, d)^{\lL^{-1} \sss}
	\stackrel{\cong}{\to} \nN_{\cC/S}(\eE_{\bullet}, d)^{\lL \sss}.
	\end{align*}
Similarly to above, the $T$-valued points 
of the right hand side 
consist of (\ref{Tvalue2}) such that 
$u' \colon (\eE_1)_T \to \pP$ is surjective. 
Also note that, under the isomorphism (\ref{isom:QuotW}), 
the diagonal torus $\mathbb{C}^{\ast} \subset \GL(W)$
acts on $\mathrm{Quot}^{\circ}_{\cC/S}(W \otimes \oO_{\cC}, d)$
trivially and acts on fibers of $\mM_{\cC/S}(\eE^i, d) \to \mM_{\cC/S}(d)$
by weight one. 
Therefore by Theorem~\ref{thm:koszul} and Lemma~\ref{lem:action}, 
the last category in (\ref{equiv:Wquot})
for $\lL^{-1}$-semistable part is equivalent to 
$D^b(\mathbf{Quot}(\hH, d))$, so  
the second equivalence in (\ref{weq:N}) follows. 
	\end{proof}

We take the 
special choices of $m_{\bullet}^{\pm}$
by $m_{i}^+=-\eta_i+\varepsilon$
for $0<\varepsilon \ll 1$ 
and $m_i^-=0$, and define 
\begin{align}\label{window:special}
	\mathbb{W}_{\rm{glob}}^+(d) \cneq \mathbb{W}_{m_i^+=-\eta_i^{+}+\varepsilon}(d), \ 
	\mathbb{W}_{\rm{glob}}^-(d) \cneq \mathbb{W}_{m_i^-=0}(d). 
	\end{align}
\begin{lem}\label{lem:tauy}
	An object $\eE \in \MF^{\C}(\nN_{\cC/S}(\eE^{\bullet}, d), w)$ lies in 
	$\mathbb{W}_{\rm{glob}}^+(d)$ 
	(resp.~$\mathbb{W}_{\rm{glob}}^-(d)$) if and only if
	for any $p \in M_{\cC/S}(d)$
as in (\ref{point:p}), the object $\iota_p^{\ast}\eE \in \MF^{\C}(\widehat{\gG}_{a, b}^{(m)}(d^{(\ast)})_R, \iota_p^{\ast}w)$
	lies in $\widehat{\mathbb{W}}_c(d^{(\ast)})$
	for $c=a$ (resp. $c=b$). 
	Here $\iota_p$ is given in (\ref{dia:formal20}) and 
	$\widehat{\mathbb{W}}_c(d^{(\ast)})$ is given in (\ref{abuse}). 
	\end{lem}
\begin{proof}
	The defining condition 
	of $\mathbb{W}_{\rm{glob}}^{\pm}(d)$ is local
	on $M_{\cC/S}(d)$, 
	so $\eE$ is an object in $\mathbb{W}_{\rm{glob}}^{\pm}(d)$ if and 
	only if $\iota_p^{\ast}\eE$ satisfies the same 
	weight condition (\ref{cond:weight})
	with respect to the $\Theta$-stratifications (\ref{theta:N})
	pulled back via $\iota_p$.  
	As we already mentioned, they coincide with 
	KN stratifications in (\ref{KN:Gm}). 
	Therefore the argument of Proposition~\ref{prop:WGD}
	implies the lemma (see Remark~\ref{rmk:Wab}). . 
	\end{proof}

\subsection{Categorified Hall product}
For a decomposition $d=d_1+d_2$, 
we define the stack over $S$
\begin{align*}
	\eE x_{\cC/S}(d_1, d_2) \colon (Sch/S) \to (Groupoid)
	\end{align*}
whose $T$-valued points consist of exact sequences 
\begin{align}\label{ex:P}
	0 \to \pP_1 \to \pP_3 \to \pP_2 \to 0
	\end{align}
where $\pP_i$ are $T$-valued points of $\mM_{\cC/S}(d_i)$
with $d_3=d$. 
The stack $\eE x_{\cC/S}(d_1, d_2)$ is a smooth Artin stack of finite type over $S$. 
Indeed we have the obvious 
evaluation morphisms 
\begin{align}\label{stack:zero0}
	\xymatrix{
\eE x_{\cC/S}(d_1, d_2) \ar[r]^-{\ev_3}\ar[d]_-{(\ev_1, \ev_2)} & \mM_{\cC/S}(d_3) \\
\mM_{\cC/S}(d_1)\times_S \mM_{\cC/S}(d_2)	
}
	\end{align}
The left vertical arrow is smooth because of 
the vanishing of $\Ext^2$ on curves, so 
$\eE x_{\cC/S}(d_1, d_2)$ is also smooth. 
Since every stacks in (\ref{stack:zero0}) are smooth over $S$ 
and $\ev_3$ is proper, 
we have the functor 
\begin{align}\label{chall1}
	\ast \cneq \ev_{3\ast}(\ev_1, \ev_2)^{\ast} \colon 
	D^b(\mM_{\cC/S}(d_1)) \boxtimes_S D^b(\mM_{\cC/S}(d_2)) \to 
	D^b(\mM_{\cC/S}(d))
	\end{align}
giving categorified Hall algebra
structure
on 
\begin{align}\label{chall1.5}
	\bigoplus_{d\ge 0} D^b(\mM_{\cC/S}(d)).
	\end{align}

We also define the stack over $S$
\begin{align*}
	\eE x_{\cC/S}(\eE^{\bullet}, d_1, d_2)
	\colon (Sch/S) \to (Groupoid)
	\end{align*}
whose $T$-valued points consist of diagrams 
\begin{align}\label{ex:Tvalue}
	\xymatrix{
	 & 0 \ar[d] & \\
0\ar[r] \ar[d] & \pP_1 \ar[r] \ar[d] & 0\ar[d] \\
\eE_T^0 \ar[r]^-{u_3} \ar@{=}[d] & \pP_3 \ar[r]^-{v_3} \ar[d] & \eE_T^{-1} \otimes \omega_{\cC_T/T}[1] \ar@{=}[d] \\
\eE_T^0 \ar[r]^-{u_2} & \pP_2 \ar[r]^-{v_2} \ar[d] & \eE_T^{-1} \otimes \omega_{\cC_T/T}[1] \\
& 0 &
}
\end{align}
where $\pP_i$ are $T$-valued points of $\mM_{\cC/S}(d_i)$
with $d_3=d$, 
the vertical arrows are exact sequences, the middle and the bottom 
arrows are $T$-valued points of 
$\nN_{\cC/S}(\eE^{\bullet}, d_i)$ for $i=3, 2$. 

The stack $\eE x_{\cC/S}(\eE^{\bullet}, d_1, d_2)$
is a smooth Artin stack of finite type. 
Indeed given an exact sequence (\ref{ex:P}), 
giving a diagram (\ref{ex:Tvalue}) is equivalent to giving 
morphisms $\eE_T^0 \to \pP_3$ 
and $\pP_2 \to \eE_T^{-1} \otimes \omega_{\cC_T/T}[1]$, so
$\eE x_{\cC/S}(\eE^{\bullet}, d_1, d_2)$ is constructed as 
a Cartesian square
\begin{align*}
	\xymatrix{
	\eE x_{\cC/S}(\eE^{\bullet}, d_1, d_2) \ar[r] \ar[d] \diasquare & \eE x_{\cC/S}(d_1, d_2) 
	\ar[d]^-{(\ev_2, \ev_3)} \\
	\mM_{\cC/S}(\eE^{-1}, d_2)^{\vee} \times_S 
	\mM_{\cC/S}(\eE^0, d_3) \ar[r] & \mM_{\cC/S}(d_2) \times_S \mM_{\cC/S}(d_3). 
}
	\end{align*}
Since $\eE x_{\cC/S}(d_1, d_2)$ is smooth and the bottom 
horizontal arrow is a vector bundle, 
the above construction in particular implies that 
$\eE x_{\cC/S}(\eE^{\bullet}, d_1, d_2)$ is smooth over $S$. 

We have the obvious evaluation morphisms, 
commuting with super-potentials in (\ref{tr:1})
\begin{align}\label{stack:zero}
	\xymatrix{
		\eE x_{\cC/S}(\eE^{\bullet}, d_1, d_2) 
		\ar[r]^-{\ev_3}\ar[d]_-{(\ev_1, \ev_2)} & \nN_{\cC/S}(\eE^{\bullet}, d_3) \ar[dd]  
		\ar[rd]^-{w}& \\
		\mM_{\cC/S}(d_1)\times_S \nN_{\cC/S}(\eE^{\bullet}, d_2)
		 \ar[d] \ar[rr]^{(0, w)} & & \mathbb{A}^1 \\
		M_{\cC/S}(d_1)\times_S M_{\cC/S}(d_2)	\ar[r]^-{\oplus}  &
	M_{\cC/S}(d_3).  &  	
	}
\end{align}
Here the right vertical arrow 
is the morphism $g$ in (\ref{dia:gmoduli}), and the left bottom 
vertical arrow is the product of (\ref{pi:MCS}) with $g$. 
We have an auxiliary $\C$-action on $\eE x_{\cC/S}(\eE^{\bullet}, d_1, d_2)$
acting on (\ref{ex:Tvalue})
by $(u_3, v_3, u_2, v_2) \mapsto (tu_3, v_3, tu_2, v_2)$. 
Then the diagram (\ref{stack:zero}) is equivariant under the 
auxiliary $\C$-actions. 
Also note that every stacks in the upper left diagram in (\ref{stack:zero}) are smooth and 
$\ev_3$ is proper as any fiber is a closed subscheme of a Quot scheme.
Therefore 
we have the categorified Hall product 
\begin{align}\label{prod:MFN}
	\ast \cneq 
	\ev_{3\ast}(\ev_1, \ev_2)^{\ast} \colon 
	D^b(\mM_{\cC/S}(d_1)) \boxtimes_S 
	\MF^{\C}(\nN_{\cC/S}(\eE^{\bullet}, d_2), w)
	\to 	\MF^{\C}(\nN_{\cC/S}(\eE^{\bullet}, d), w)
	\end{align}
which gives a left module structure of (\ref{chall1.5}) on 
\begin{align*}
	\bigoplus_{d\ge 0} \MF^{\C}(\nN_{\cC/S}(\eE^{\bullet}, d), w). 
	\end{align*}
\subsection{Semiorthogonal decomposition}
Recall that $r=\rank(\eE)$, 
$a=\rank\eE^0$, $b=\rank \eE^{-1}$ so that $r=a-b$. 
We prove the following proposition: 
\begin{prop}\label{prop:main}
	For $(d_1, \ldots, d_r) \in \mathbb{Z}_{\ge 0}^r$, 
	the categorified Hall product (\ref{prod:MFN}) 
	induces the fully-faithful 
	functor 
	\begin{align}\notag
		\ast \colon 
		\mathbb{W}_{\rm{glob}}(d_1) \boxtimes_S
		(\mathbb{W}_{\rm{glob}}(d_2)\otimes \lL)
	\boxtimes_S \cdots \boxtimes_S 
	(\mathbb{W}_{\rm{glob}}(d_r)\otimes \lL^{r-1})
&\boxtimes_S	(\mathbb{W}^{-}_{\rm{glob}}(d-\sum_{i=1}^r d_i)\otimes \lL^{r}) \\
\label{funct:ast}&\to \mathbb{W}_{\rm{glob}}^+(d)
\end{align}	
such that, by setting $\cC_{\rm{glob}}^-(d_{\bullet})$ to 
be the essential image of the above functor, 
we have the semiorthogonal decomposition 
\begin{align*}
	\mathbb{W}_{\rm{glob}}^+(d)
	=\langle \cC_{\rm{glob}}^-(d_{\bullet}) : 
	d_{\bullet} \in \mathbb{Z}_{\ge 0}^r \rangle. 
	\end{align*}
Here $\Hom(\cC_{\rm{glob}}(d_{\bullet}'), \cC_{\rm{glob}}(d_{\bullet}))=0$
for $d_{\bullet}' \succ d_{\bullet}$. 
\end{prop}
\begin{proof}
	Following the argument of~\cite[Theorem~5.16]{Toconi}, 
	it is enough to show that the functor (\ref{funct:ast})
	is fully-faithful and forms a semiorthogonal decomposition 
	formally locally on $M_{\cC/S}(d)$. 
	Let us take a closed point $p \in M_{\cC/S}(d)$
	written as in (\ref{point:p}), and set $R=\widehat{\oO}_{S, s}$. 
	We write the formal completion $\widehat{M}_{\cC/S}(d)_p$
	as $\widehat{M}_{\cC/S}(d)_{d^{(\ast)}p^{(\ast)}}$. 
	Let $d=d_1+d_2$ be a decomposition. 
	Then we have the Cartesian square 
	\begin{align*}
		\xymatrix{
		\coprod_{d_1^{(\ast)}+d_2^{(\ast)}=d^{(\ast)},
			\lvert d_i^{(\ast)}\rvert=d_i}
	\widehat{M}_{\cC/S}(d_1)_{d_1^{(\ast)}p^{(\ast)}} \widehat{\times}_R
	\widehat{M}_{\cC/S}(d_2)_{d_2^{(\ast)}p^{(\ast)}}
	\ar[r] \ar[d] & \widehat{M}_{\cC/S}(d)_{d^{(\ast)}p^{(\ast)}} \ar[d] \\
	M_{\cC/S}(d_1) \times_S M_{\cC/S}(d_2) \ar[r]^-{\oplus} & M_{\cC/S}(d). 
}
		\end{align*}
Together with the diagrams (\ref{dia:formal10}), (\ref{dia:formal20}), 
the base change of the diagram (\ref{stack:zero})
via $\widehat{M}_{\cC/S}(d)_p \to M_{\cC/S}(d)$
gives the diagram 
 	\begin{align}\notag
 	\xymatrix{\widehat{\eE x}(\eE^{\bullet}, d_1, d_2)_p
 		\ar[r] \ar[d]   
 		& \widehat{\gG}_{a, b}^{(m)}(d^{(\ast)})_R \ar[dd] \ar[rd]^-{\iota_p^{\ast}w} &\\
 		\coprod_{d_1^{(\ast)}+d_2^{(\ast)}=d^{(\ast)},
 		\lvert d_i^{(\ast)}\rvert=d_i}	\widehat{\mM}_{Q^{(m)}}(d_1^{(\ast)})_R \widehat{\times}_R
 		\widehat{\gG}_{a, b}^{(m)}(d_2^{(\ast)})_R
 		\ar[d] \ar[rr]^-{(0, \iota_p^{\ast}w)}
 		& & \mathbb{A}^1 \\
 		\coprod_{d_1^{(\ast)}+d_2^{(\ast)}=d^{(\ast)},
 		\lvert d_i^{(\ast)}\rvert=d_i}	\widehat{M}_{Q^{(m)}}(d_1^{(\ast)})_R \widehat{\times}_R
 		\widehat{M}_{Q^{(m)}}(d_2^{(\ast)})_R
 		\ar[r]^-{\oplus}
 		& \widehat{M}_{Q^{(m)}}(d^{(\ast)})_R. &
 	}
 \end{align}
By comparing with the diagram (\ref{dia:formal2}), we see that the 
base change of the product functor (\ref{prod:MFN}) via $\widehat{M}_{\cC/S}(d)_p \to M_{\cC/S}(d)$ is identified with 
the direct sum of (\ref{Hprod:hat}) for all the 
decompositions $d^{(\ast)}=d_1^{(\ast)}+d_2^{(\ast)}$ with 
$\lvert d_i^{(\ast)} \rvert=d_i$. 

For $(d_1, \ldots, d_r) \in \mathbb{Z}_{\ge 0}^r$, the above 
argument shows that the base change of the product functor 
\begin{align*}
\ast \colon	D^b(\mM_{\cC/S}(d_1)) \boxtimes_S \cdots \boxtimes_S
	D^b(\mM_{\cC/S}(d_r)) &\boxtimes_S \MF^{\C}(\nN_{\cC/S}(\eE^{\bullet}, d-\sum_{i=1}^r d_i), 
	w) \\
	&\to \MF^{\C}(\nN_{\cC/S}(\eE^{\bullet}, d), 
	w)
	\end{align*}
via 
$\widehat{M}_{\cC/S}(d)_p \to M_{\cC/S}(d)$
gives the Hall product in Subsection~\ref{subsec:someversion}
\begin{align*}
	\ast \colon 
	\bigoplus_{\begin{subarray}{c}
		(d_1^{(\ast)}, \ldots, d_r^{(\ast)}) \\
		\lvert d_i^{(\ast)} \rvert=d_i
	\end{subarray}}
	D^b(\widehat{\mM}_{Q^{(m)}}(d_1^{(\ast)})_R) \widehat{\boxtimes}_R
	\cdots \widehat{\boxtimes}_R 	D^b(\widehat{\mM}_{Q^{(m)}}(d_r^{(\ast)})_R) 
	&\widehat{\boxtimes}_R \MF^{\C}(\widehat{\gG}_{a, b}^{(m)}(d^{(\ast)}-\sum_{i=1}^r d_i^{(\ast)})_R, \iota_p^{\ast}w) \\
	&\to \MF^{\C}(\widehat{\gG}_{a, b}^{(m)}(d^{(\ast)})_R, \iota_p^{\ast}w).
	\end{align*}
Then from Lemma~\ref{lem:loc1} and Lemma~\ref{lem:tauy}, 
the required formal local 
statement is given in Theorem~\ref{thm:SOD2}. 
	\end{proof}

The following is the main result in this paper: 
\begin{thm}\label{thm:main}
	There is a semiorthogonal decomposition of the form 
	\begin{align*}
		&D^b(\mathbf{Quot}_{\cC/S}(\eE, d))
		= \\
		&\left\langle D^b(\mathrm{Sym}^{d_1}(\cC/S) \times_S
		\cdots \times_S \mathrm{Sym}^{d_r}(\cC/S) \times_S
		\mathbf{Quot}_{\cC/S}(\hH, d-\sum_{i=1}^r d_i) ) :
		(d_1, \ldots, d_r)\in \mathbb{Z}_{\ge 0}^r \right\rangle. 
	\end{align*}
\end{thm}
\begin{proof}
	The theorem follows from Proposition~\ref{prop:window:eq}, 
	Proposition~\ref{prop:main}, 
	and the fact that $\mathbb{W}_{\rm{glob}}(d)$ is equivalent to 
	$D^b(\mathrm{Sym}^d(\cC/S))$ by its definition. 
	\end{proof}
\begin{rmk}\label{rmk:boson}
	The number of semiorthogonal summands in Theorem~\ref{thm:main}
	involving $\mathbf{Quot}_{\cC/S}(\hH, d-k)$ is 
	$\dim \mathrm{Sym}^k(\mathbb{C}^r)$, while 
	that involving $\mathbf{Quot}_{S}(\hH, d-k)$ in (\ref{quot:formula})
	is $\dim \mathrm{Sym}^k(\mathbb{C}^r[1])$. 
	So in some sense the semiorthogonal decompositions in Theorem~\ref{thm:main}
	and (\ref{quot:formula}) are related by boson-fermion correspondence 
	(see the related phenomena for cohomological DT theory in~\cite{Davbos}).
\end{rmk}

For a smooth projective curve $C$, 
the Quot scheme of zero-dimensional quotients $\Quot_C(\eE, d)$
is smooth of expected dimension $rd$. 
Therefore Theorem~\ref{thm:main} implies the following: 
\begin{cor}\label{cor:quotC}
	There is a semiorthogonal decomposition of the form 
	\begin{align}\label{quotC2}
		D^b(\Quot_C(\eE, d))=\left\langle  
		D^b(\mathrm{Sym}^{d_1}(C) \times \cdots \times \mathrm{Sym}^{d_r}(C)) :
		d_1+\cdots+ d_r=d
		\right\rangle. 
		\end{align}
	\end{cor}

	\begin{rmk}\label{rmk:sym}
	The semiorthogonal decomposition (\ref{quotC2}) may be further decomposed. Indeed for each $\delta>0$, there is a 
	semiorthogonal decomposition proved in~\cite[Corollary~5.11]{TodDK}
	\begin{align*}
		D^b(\mathrm{Sym}^{g-1+\delta}(C))=\langle D^b(\mathrm{Sym}^{g-1-\delta}(C)), 
		\overbrace{J(C), \ldots, J(C)}^{\delta})
	\end{align*}
	where $g$ is the genus of $C$ and $J(C)$ is the Jacobian of $C$
	So if 
	$d_i>g-1$ for some $i$ in (\ref{intro:sodQ}), it can be further decomposed. 
	In particular if $C=\mathbb{P}^1$, then (\ref{quotC2}) implies the 
	existence of a full exceptional collection whose cardinality is  
	\begin{align*}
		\sum_{d_1+\cdots+d_r=d} \prod_{i=1}^r (d_i+1)
		=\binom{2r+d-1}{d}. 
	\end{align*}	
\end{rmk}

			\bibliographystyle{amsalpha}
	\bibliography{math}

\providecommand{\bysame}{\leavevmode\hbox to3em{\hrulefill}\thinspace}
\providecommand{\MR}{\relax\ifhmode\unskip\space\fi MR }
\providecommand{\MRhref}[2]{%
  \href{http://www.ams.org/mathscinet-getitem?mr=#1}{#2}
}
\providecommand{\href}[2]{#2}
\begin{thebibliography}{BFK19}

\bibitem[Alp13]{MR3237451}
J.~Alper, \emph{Good moduli spaces for {A}rtin stacks}, Ann. Inst. Fourier
  (Grenoble) \textbf{63} (2013), no.~6, 2349--2402.

\bibitem[BFK19]{MR3895631}
M.~Ballard, D.~Favero, and L.~Katzarkov, \emph{Variation of geometric invariant
  theory quotients and derived categories}, J. Reine Angew. Math. \textbf{746}
  (2019), 235--303.

\bibitem[BFP20]{BFP}
M.~Bagnarol, B.~Fantechi, and F.~Perroni, \emph{On the motive of {Q}uot schemes
  of zero-dimensional quotients on a curve}, New York J. Math. \textbf{26}
  (2020), 138--148.

\bibitem[Bif89]{Bifet}
E.~Bifet, \emph{Sur les points fixes du sch\'{e}ma {${\rm
  Quot}_{{\mathcal{O}}^r_X/X/k}$} sous l'action du tore
  {${\mathbf{G}}^r_{m,k}$}}, C. R. Acad. Sci. Paris S\'{e}r. I Math.
  \textbf{309} (1989), no.~9, 609--612.

\bibitem[BO]{B-O2}
A.~Bondal and D.~Orlov, \emph{Semiorthogonal decomposition for algebraic
  varieties}, arXiv:9506012.

\bibitem[Dav]{Davbos}
B.~Davison, \emph{A boson-fermion correspondence in cohomological
  {D}onaldson-{T}homas theory}, arXiv:2109.09788.

\bibitem[DM20]{DaMe}
B.~Davison and S.~Meinhardt, \emph{Cohomological {D}onaldson-{T}homas theory of
  a quiver with potential and quantum enveloping algebras}, Invent. Math.
  \textbf{221} (2020), no.~3, 777--871.

\bibitem[EP15]{MR3366002}
A.~I. Efimov and L.~Positselski, \emph{Coherent analogues of matrix
  factorizations and relative singularity categories}, Algebra Number Theory
  \textbf{9} (2015), no.~5, 1159--1292.

\bibitem[Hir17]{MR3631231}
Y.~Hirano, \emph{Derived {K}n\"orrer periodicity and {O}rlov's theorem for
  gauged {L}andau-{G}inzburg models}, Compos. Math. \textbf{153} (2017), no.~5,
  973--1007.

\bibitem[HLa]{HalpK3}
D.~Halpern-Leistner, \emph{The {D}-equivalence conjecture for moduli spaces of
  sheaves on a {K}3 surface}, available in
  http://www.math.columbia.edu/~danhl/.

\bibitem[HLb]{Halpinstab}
\bysame, \emph{On the structure of instability in moduli theory},
  arXiv:1411.0627.

\bibitem[HL15]{MR3327537}
\bysame, \emph{The derived category of a {GIT} quotient}, J. Amer. Math. Soc.
  \textbf{28} (2015), no.~3, 871--912.

\bibitem[HLS20]{HLKSAM}
D.~Halpern-Leistner and S.~V. Sam, \emph{Combinatorial constructions of derived
  equivalences}, J. Amer. Math. Soc. \textbf{33} (2020), no.~3, 735--773.

\bibitem[Isi13]{MR3071664}
M.~U. Isik, \emph{Equivalence of the derived category of a variety with a
  singularity category}, Int. Math. Res. Not. IMRN (2013), no.~12, 2787--2808.

\bibitem[Jia]{JiangQuot}
Q.~Jiang, \emph{Derived categories of {Q}uot schemes of locally free quotients
  {I}}, arXiv:2107.09193.

\bibitem[JT17]{MR3607000}
Y.~Jiang and R.~Thomas, \emph{Virtual signed {E}uler characteristics}, J.
  Algebraic Geom. \textbf{26} (2017), no.~2, 379--397.

\bibitem[Kap84]{Kapranov}
M.~Kapranov, \emph{Derived category of coherent sheaves on {G}rasssmann
  manifolds}, (Russian) Izv.Akad.Nauk SSSR Ser.Mat \textbf{48} (1984),
  192--202.

\bibitem[Kaw02]{Ka1}
Y.~Kawamata, \emph{${D}$-equivalence and ${K}$-equivalence},
  J.~Differential~Geom.~ \textbf{61} (2002), 147--171.

\bibitem[KS11]{MR2851153}
M.~Kontsevich and Y.~Soibelman, \emph{Cohomological {H}all algebra, exponential
  {H}odge structures and motivic {D}onaldson-{T}homas invariants}, Commun.
  Number Theory Phys. \textbf{5} (2011), no.~2, 231--352.

\bibitem[KT21]{koTo}
N.~Koseki and Y.~Toda, \emph{Derived categories of {T}haddeus pair moduli
  spaces via d-critical flips}, Adv. Math. \textbf{391} (2021), Paper No.
  107965, 55.

\bibitem[MR10]{MrRi}
I.~Mirkovi\'{c} and S.~Riche, \emph{Linear {K}oszul duality}, Compos. Math.
  \textbf{146} (2010), no.~1, 233--258.

\bibitem[OP]{OpPan}
D.~Oprea and R.~Pandharipande, \emph{Quot schemes of curves and surfaces:
  virtual classes, integrals, {E}uler characteristics}, arXiv:1903.08787.

\bibitem[Orl09]{Orsin}
D.~Orlov, \emph{Derived categories of coherent sheaves and triangulated
  categories of singularities}, Algebra, arithmetic, and geometry: in honor of
  {Y}u. {I}. {M}anin. {V}ol. {II}, Progr. Math., vol. 270, Birkh\"{a}user
  Boston, Inc., Boston, MA, 2009, pp.~503--531.

\bibitem[Orl12]{Ornonaff}
D.~Orlov, \emph{Matrix factorizations for nonaffine {LG}-models}, Math. Ann.
  \textbf{353} (2012), no.~1, 95--108.

\bibitem[P{\u{a}}da]{Tudor1.5}
T.~P{\u{a}}durariu, \emph{Categorical and {K}-theoretic {H}all algebras for
  quivers with potential}, arXiv:2107.13642.

\bibitem[P{\u{a}}db]{Tudor1.7}
\bysame, \emph{Generators for {K}-theoretic {H}all algebras of quivers with
  potential}, arXiv:2108.07919.

\bibitem[PT]{PTzero}
T.~P\u{a}durariu and Y.~Toda, \emph{Categorical and {K}-theoretic
  {D}onaldson-{T}homas theory of $\mathbb{C}^3$ (part {I})}, arXiv:2207.01899.

\bibitem[PV11]{MR3112502}
A.~Polishchuk and A.~Vaintrob, \emph{Matrix factorizations and singularity
  categories for stacks}, Ann. Inst. Fourier (Grenoble) \textbf{61} (2011),
  no.~7, 2609--2642.

\bibitem[Shi12]{MR2982435}
I.~Shipman, \emph{A geometric approach to {O}rlov's theorem}, Compos. Math.
  \textbf{148} (2012), no.~5, 1365--1389.

\bibitem[Toda]{Toddbir}
Y.~Toda, \emph{Birational geometry for d-critical loci and wall-crossing in
  {C}alabi-{Y}au 3-folds}, to appear in Algebraic Geometry, arXiv:1805.00182.

\bibitem[Todb]{TocatDT}
\bysame, \emph{Categorical {D}onaldson-{T}homas theory for local surfaces},
  arXiv:1907.09076.

\bibitem[Todc]{Toconi}
\bysame, \emph{Categorical wall-crossing formula for {D}onaldson-{T}homas
  theory on the resolved conifold}, arXiv:2109.07064.

\bibitem[Todd]{ToQuot}
\bysame, \emph{Derived categories of {Q}uot schemes of locally free quotients
  via categorified {H}all products}, arXiv:2110.02469.

\bibitem[Tod21]{TodDK}
\bysame, \emph{Semiorthogonal decompositions of stable pair moduli spaces via
  d-critical flips}, J. Eur. Math. Soc. (JEMS) \textbf{23} (2021), no.~5,
  1675--1725.

\end{thebibliography}

	\vspace{5mm}
	
	Kavli Institute for the Physics and 
	Mathematics of the Universe (WPI), University of Tokyo,
	5-1-5 Kashiwanoha, Kashiwa, 277-8583, Japan.

	\textit{E-mail address}: yukinobu.toda@ipmu.jp
		
	\end{document}